\documentclass[reqno,11pt]{amsart}

\usepackage{amsmath, amsthm, amssymb, amsfonts, amsthm}
\usepackage{amsxtra, amscd, geometry, graphicx,epic,eepic}
\usepackage[all,cmtip]{xy}
\usepackage{mathrsfs}
\usepackage{hyperref,verbatim}
\usepackage{color}

\newtheorem{proposition}{Proposition}
\newtheorem{cor}[proposition]{Corollary}
\newtheorem{lemma}[proposition]{Lemma}

\theoremstyle{definition}

\newtheorem*{rmk*}{Remark}

\numberwithin{equation}{section}

\renewcommand{\tilde}[1]{\widetilde{#1}}%

\newcommand{\XX}{\mathcal X}

\newcommand{\N}{\mathbb N}
\newcommand{\Q}{\mathbb Q}
\newcommand{\Z}{\mathbb Z}
\newcommand{\C}{\mathbb C}
\newcommand{\R}{\mathbb R}
\renewcommand{\H}{\mathbb{H}}

\newcommand{\F}{\mathbb{F}}
\renewcommand{\Re}{\operatorname{Re}}

\newcommand{\llangle}{\langle\! \langle}
\newcommand{\rrangle}{\rangle\! \rangle}
\newcommand{\tT}{\tilde{T}}
\newcommand{\bT}{\bar{T}}
\newcommand{\tOmega}{\tilde{\Omega}}
\newcommand{\SSS}{\mathcal{S}}

\newcommand{\AAA}{\mathcal{A}}
\newcommand{\gplus}{\gamma_{\infty}}
\newcommand{\gneg}{\gamma_{-\infty}}
\newcommand{\ST}{ST^{-1}}
\newcommand{\MM}{\mathcal M}
\newcommand{\balpha}{\bar{\alpha}}
\newcommand{\e}{\epsilon}
\newcommand{\rL}{\mathbb{L}}
\newcommand{\bL}{\mathbf{L}}
\newcommand{\rR}{\mathbb{R}}
\newcommand{\bR}{\mathbf{R}}

\newcommand{\x}{\xi_\gamma}

\title[Coding of geodesics, odd and even continued fractions]{Coding of geodesics on some modular surfaces and
applications to odd and even continued fractions}

\author{Florin P. Boca}
\author{Claire Merriman}

\address{Department of Mathematics, University of Illinois at Urbana-Champaign,
Urbana, IL 61801}

\address{E-mail: fboca@math.uiuc.edu}

\address{E-mail: emerrim2@illinois.edu}

\begin{document}

\date{\today}

\begin{abstract}
The connection between geodesics on the modular surface $\operatorname{PSL}(2,\Z)\backslash \H$ and regular
continued fractions, established by Series, is extended to a connection between geodesics on $\Gamma\backslash \H$ and odd
and grotesque continued fractions,
where $\Gamma\cong \Z_3 \ast \Z_3$ is the index two subgroup of $\operatorname{PSL}(2,\Z)$ generated by the free elements
of order three $\left( \begin{smallmatrix} 0 & -1 \\ 1 & 1 \end{smallmatrix} \right)$ and
$\left( \begin{smallmatrix} 0 & 1 \\ -1 & 1 \end{smallmatrix} \right)$, and having an ideal quadrilateral as fundamental domain.
A similar connection between geodesics on $\Theta\backslash \H$ and even continued fractions is discussed
in our framework, where $\Theta$ denotes the Theta subgroup of $\operatorname{PSL}(2,\Z)$ generated by
$\left( \begin{smallmatrix} 0 & -1 \\ 1 & 1 \end{smallmatrix} \right)$ and
$\left( \begin{smallmatrix} 1 & 2 \\ 0 & 1 \end{smallmatrix} \right)$.
\end{abstract}

\maketitle

\section{Introduction}

The connection between geodesics on the modular surface $\MM =\operatorname{PSL}(2,\Z)\backslash \H$
and regular continued fractions (RCF),
originating in the seminal work of Artin \cite{Art}, has generated a significant amount of interest.
Inspired by earlier work of Moeckel \cite{Moe}, Series \cite{Ser} established explicit connections between the geodesic
flow on $\MM$, geodesic coding, and RCF dynamics. The way ideal triangles of the Farey tessellation $\F$ cut oriented
geodesics $\gamma$ on the upper half-plane $\H$ play a central role in this approach.
The geodesic crosses two sides of a triangle,
and the geodesic arc is labeled $L$ or $R$ according to whether the vertex shared by the sides is to the left
or right, respectively, of the geodesic. This labeling is invariant under $\Gamma (1)=\operatorname{PSL}(2,\Z)$,
and hence under any of its subgroups.
These geodesics $\gamma$ are lifts of geodesics
$\bar{\gamma}$ on $\MM$, which are uniquely determined
by infinite two-sided cutting sequences $ \ldots L^{n_{-1}} R^{n_0} L^{n_1} \ldots$. The sequence of positive
integers $(n_i)_{i=-\infty}^\infty$ is intimately related with regular continued fractions, since
\begin{equation*}
\gamma_{-\infty}=\frac{-1}{n_0+\frac{1}{n_{-1}+\frac{1}{n_{-2}+\cdots}}} =
-[ n_0,n_{-1},n_{-2},\ldots ]\ \ \mbox{\rm and} \ \
\gamma_{\infty}=n_1 +\frac{1}{n_2+\frac{1}{n_3 + \cdots}} =
[ n_1;n_2,n_3,\ldots ]
\end{equation*}
are the negative and positive endpoints of some lift $\gamma$ of $\bar{\gamma}$ to $\H$.
Shifting along the cutting sequences is related to
the Gauss map $T$ on
$[0,1)$ and its natural extension $\bT$ on $[0,1)^2$, defined by
\begin{equation*}
T\big( [n_1,n_2,n_3,\ldots]\big) =[n_2,n_3,n_4,\ldots ],
\end{equation*}
\begin{equation*}
\bT \big( [n_1,n_2,n_3,\ldots ],[n_0,n_{-1},n_{-2},\ldots ]\big) =
\big( [n_2,n_3,n_4,\ldots ],[n_1,n_0,n_{-1},\ldots ]\big) .
\end{equation*}
Some different approaches for coding the geodesic flow on $\MM$ were considered by Arnoux in \cite{Arn}
and by Katok and Ugarcovici \cite{KU1}.

A large class of continued fractions has been studied in the context of
the geodesic flow and symbolic dynamics. A non-exhaustive list includes backward continued fractions
\cite{AF1,AF2}, even continued fractions \cite{AaD,BL,Cel,KU1}, Rosen continued fractions \cite{MS},
$(a,b)$-continued fractions \cite{KU2}, Nakada $\alpha$-continued
fractions and $\alpha$-Rosen continued fractions \cite{AS},
or other classes of complex or Heisenberg continued fractions \cite{LV}.
In different directions, the symbolic dynamics associated with the
billiard flow on modular surfaces of uniform triangle groups, and with the
geodesic flow on two-dimensional hyperbolic good orbifolds, have been thoroughly
investigated by Fried \cite{Fr} and Pohl \cite{Po}.

This note describes codings of geodesics on the modular surfaces $\MM_o$ and $\MM_e$,
associated with subgroups of index two and three in $\Gamma (1)$, respectively.
The coding of $\MM_o$ is hereby connected, in the spirit of \cite{Ser}, to the dynamics of odd and grotesque continued fractions
(OCF and GCF respectively),
and the coding of $\MM_e$ is connected to the even (ECF) and extended even continued fractions.
These 
continued fractions were first investigated by
Rieger \cite{Rie} and by Schweiger \cite{Sch1,Sch2}.

We consider the modular surface $\MM_o =\Gamma \backslash \H=\pi_o (\H)$, where $\Gamma$ is the
index two subgroup of $\Gamma (1)=\operatorname{PSL}(2,\Z)$ generated by the M\" obius
transformations $S(z)=\frac{-1}{z+1}$ and $T(z)=z+2$ acting on $\H$. Equivalently,
$\Gamma$ is generated by the order three matrices
$S=\left( \begin{smallmatrix} 0 & -1 \\ 1 & 1 \end{smallmatrix}\right)$
and $ST^{-1}=\left( \begin{smallmatrix} 0 & 1 \\ -1 & 1 \end{smallmatrix} \right)$.
The corresponding Dirichlet fundamental domain with respect to $i$ is
the quadrilateral ${\mathfrak F}$ bounded by the geodesic arcs
$[0,\omega]$, $[\omega,\infty]$, $[0,\omega^2]$ and $[\omega^2,\infty]$,
where $\omega=\frac{1}{2} (1+i\sqrt{3})$ (see Figure \ref{Figure1}).
The transformation $S$ fixes $\omega^2=\omega-1$ and cyclically permutes the points
$\infty, 0, -1$, and respectively $i,\frac{-1+i}{2},-1+$.
On the other hand $ST^{-1}$ fixes $\omega$ and permutes $\infty, 0, 1$, and
$1,\frac{1+i}{2},1+i$ (see Figures \ref{Figure1} and \ref{Figure3}).
As shown by Lemma \ref{lemma7}, the point $\pi_o(\infty)$ is the only cusp of
$\MM_o$, while $\pi_o(i)$ is a regular point for $\MM_o$ since we deal with
a two-fold cover ramified at $i$ which makes the singularity disappear.

The parts of the fundamental region ${\mathfrak F}$ on either side of the imaginary axis
are considered separately.
First, we consider the triangle $\omega^2,0,\infty$. The union of the images of this triangle under $I, S,$ and $S^2$
gives the ideal triangle $-1,0,\infty$. Similarly, the triangle
with vertices $\omega, 0, \infty$, under $I, \ST$ and $(\ST)^2$ is the ideal triangle $1,0,\infty$.
Together, these regions form the ideal quadrilateral $\Delta$ with vertices $-1,0,1$ and $\infty$.
The images of $\Delta$ under $\Gamma$ form the Farey tessellation. That is, two rational numbers
$\frac{p}{q},\frac{p'}{q'}$ are joined by a side of the Farey tessellation precisely when $pq'-p'q=\pm 1$.

\begin{figure}
\centering\includegraphics[scale=0.8, bb=100 10 260 160]{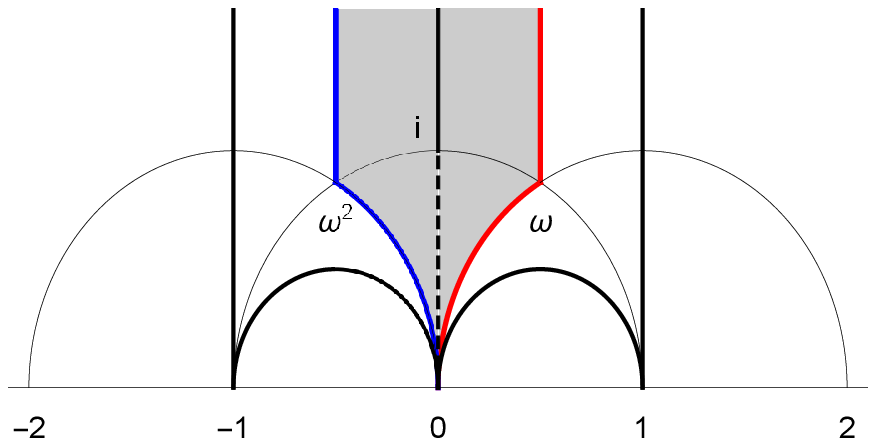}
\includegraphics[scale=0.9, bb=0 0 80 120]{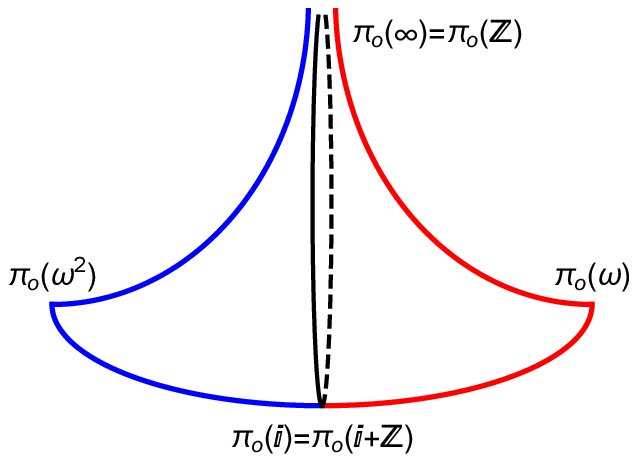}
\caption{The fundamental region ${\mathfrak F}$ and the modular surface $\MM_o=\Gamma \backslash \H=\pi_o (\H)$}
\label{Figure1}
\end{figure}

\begin{figure}
\centering\includegraphics[width=12cm]{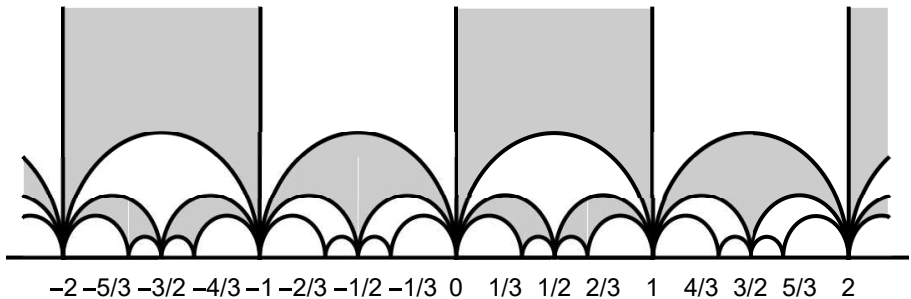}
\caption{The checkered Farey tessellation}
\label{Figure2}
\end{figure}

This is the same tessellation considered by Series \cite{Ser}, but we add a checkerboard coloring,
as shown in Figure \ref{Figure2}.
The triangle $-1, 0, \infty$ is light, while $1,0,\infty$ is dark, then continue in a checkerboard pattern,
so that each of the three neighboring Farey cells of a light cell are dark, and vice versa.
We code oriented geodesics by including the shade of the Farey cell.
Concretely, a light $L$ is denoted by $\rL$, a dark $L$ by $\bL$,
a light $R$ by $\rR$, and a dark $R$ by $\bR$. This way, every geodesic
in $\H$ with irrational endpoints is assigned an infinite two-sided sequence of symbols
$\rL$, $\bL$, $\rR$ and $\bR$. We also require that a light letter $\rL$ or $\rR$
can only be succeeded and preceded by a dark one, and vice versa.

The coding of geodesics on $\MM_o$ is
described by concatenating words of the following type:
\begin{equation}\label{string_types_o}
 (\rL \bL)^{k-1} \rL \bR  ,\quad
 (\rL \bL)^{k-1} \rR ,\quad
 (\bR \rR)^{k-1} \bL ,\quad
 (\bR \rR)^{k-1} \bR \rL ,\qquad k\geqslant 1.
\end{equation}
Single letter strings are allowed. We require colors of individual letters to alternate.
Thus, strings of the first and third type must be succeeded by those of the first or second.

A powerful tool in the study of endomorphisms in ergodic theory is provided by the natural
extension, an invertible transformation that dilates the given endomorphism and preserves many
of its ergodic properties, such as ergodicity or mixing \cite{CFS}.
Here, we describe the natural extension of the Gauss type OCF map (or, by reversing order,
of the Gauss type GCF map) as a factor of a certain explicit cross-section of the geodesic flow on $\MM_o$.
In particular, this provides a direct geometric proof for the ergodicity of this map and
allows us to recapture the invariant measures for the Gauss type OCF and GCF maps.
Other applications include a characterization of quadratic surds in terms of their OCF expansion
and their conjugate GCF expansion, as well as a tail-equivalence type description of the orbits of the action
of $\Gamma$ on the real line.

We also observe that similar results can be
obtained for even continued fractions, using the Farey tessellation without coloring
and the modular surface $\MM_e=\Theta \backslash \H$, where $\Theta$ denotes the
index three Theta subgroup in $\Gamma(1)$ generated by the transformations
$S(z)=-\frac{1}{z}$ and $T(z)=z+2$.

\section{Odd, grotesque, and even continued fractions}\label{fractiontypes}
In this section we review some properties and dynamics of odd, grotesque and even continued fractions. Odd and even continued fractions are part of the broader class of
$D$-continued fractions introduced in \cite{Kra} (see also \cite{HK,Mas,DHKM}), while the GCF is the dual algorithm
of the odd continued fractions. Instead of giving a compact presentation, we chose to be more repetitive, in order to
clarify notation and the difference between these three classes of continued fractions.

\subsection{Odd continued fractions}\label{ocf}
The OCF expansion of a number $x\in[0,1]\setminus \Q$ is given by
\begin{equation}\label{eq2}
x=[\![(a_1,\e_1), (a_2, \e_2), (a_3, \e_3), ...]\!]_o = \cfrac{1}{a_1+\cfrac{\e_1}{a_2+\cfrac{\e_2}{a_3+\cdots}}},
\end{equation}
where $\e_i =\e_i(x)\in \{\pm1\}$, $a_i =a_i(x)\in 2\N -1$ and $a_i + \e_i \geqslant 2$.
Such an expansion is unique. We also consider
\begin{equation*}
x=[\![(a_0,\e_0);(a_1,\e_1),(a_2,\e_2),\ldots ]\!]_o :=a_0+\e_0 [\![(a_1,\e_1),(a_2,\e_2),\ldots ]\!]_o \in [1,\infty),
\end{equation*}
with $\e_0\in \{ \pm 1\}$, $a_0 \in 2\N -1$ and $a_0+\e_0 \geqslant 2$, so that
$x\in (a_0,a_0+1)$ if $\e_0=1$ and $x\in (a_0-1,a_0)$ if $\e_0=-1$.

The odd Gauss map $T_o$ acts on $[0,1]$ by
\begin{equation*}
T_o(x)= \e \bigg( \frac{1}{x}-2k+1\bigg)
\quad
\mbox{\rm if} \quad
x\in B(\e,k) :=\begin{cases} \big( \frac{1}{2k},\frac{1}{2k-1}\big)
& \mbox{\rm if $\e=1$, $k\geqslant 1$} \\
\big( \frac{1}{2k-1},\frac{1}{2k-2}\big) & \mbox{\rm if $\e=-1$, $k\geqslant 2.$}
\end{cases}
\end{equation*}
Symbolically, $T_o$ acts on the OCF representation by removing the leading digit $(a_1,\e_1)$ of $x$, i.e.
\begin{equation*}
T_o \big( [\![(a_1,\e_1),(a_2,\e_2),(a_3,\e_3),\ldots]\!]_o\big) =[\![ (a_2,\e_2),(a_3,\e_3),(a_4,\e_4),\ldots ]\!]_o .
\end{equation*}
The probability measure $d\mu_o (x)=\frac{1}{3\log G} ( \frac{1}{G-1+x}+\frac{1}{G+1-x}) dx$
on $[0,1]$ is $T_o$-invariant (see \cite{Rie,Sch1}), where we denote
\begin{equation*}
G=\tfrac{1}{2}(\sqrt{5}+1).
\end{equation*}

\subsection{Grotesque continued fractions}\label{gcf}
Rieger's GCF representation of an irrational $y\in I_G:=(G-2,G)$ is given by
\begin{equation}\label{eq3}
y=\llangle (b_0,\e_0), (b_1, \e_1), (b_2, \e_2), ...\rrangle_o =
\cfrac{\e_0}{b_0+\cfrac{\e_1}{b_1+\cfrac{\e_2}{b_2+\cdots}}},
\end{equation}
where $\e_i \in \{\pm1\}$, $b_i=b_i(y) \in 2\N -1$ and $b_i + \e_i \geqslant 2$.
Every irrational number $y\in I_G$ can be uniquely represented as above, by
taking $\e_0=\e_0(y)=\operatorname{sign} y$ and $b_0=b_0(y)$ the unique odd
positive integer with $\frac{1}{\lvert y\rvert}-G \leqslant b_0(y) \leqslant
\frac{1}{\lvert y\rvert}-G+2$. The corresponding Gauss map $\tau_o$ acts
on $I_G$ by
\begin{equation*}
\tau_o (y)=\frac{\e_0(y)}{y}-b_0(y)=\frac{1}{\lvert y\rvert}-b_0(y) ,
\end{equation*}
or on the symbolic representation \eqref{eq3} by
\begin{equation*}
\tau_o \left( \llangle (b_0,\e_0),(b_1,\e_1),(b_2,\e_2),\ldots \rrangle_o \right)
=\llangle (b_1,\e_1),(b_2,\e_2),(b_3,\e_3),\ldots \rrangle_o ,
\end{equation*}
and $d\nu_o (y)=\frac{1}{3\log G} \cdot \frac{dy}{y+1}$ provides
a $\tau_o$-invariant probability measure (see \cite{Rie,Seb}).

Consider $\Omega_o =(0,1)\times I_G$.
The natural extension of $T_o$ can be realized as the invertible map
\begin{equation*}
\bT_o :\Omega_o \rightarrow \Omega_o , \qquad \bT_o (x,y) =\bigg( T_o (x),\frac{\e_1(x)}{a_1(x)+y}\bigg),
\end{equation*}
and it acts on $\Omega_o \cap (\R\setminus\Q)^2$ as a two-sided shift as follows:
\begin{equation*}
\begin{split}
\bT_o &  \big([\![(a_1,\e_1),(a_2,\e_2),(a_3,\e_3),\ldots ]\!]_o ,\llangle
(a_0,\e_0),(a_{-1},\e_{-1}),(a_{-2},\e_{-2}),\ldots \rrangle_o \big)
\\ & = \big([\![(a_2,\e_2),(a_3,\e_3),(a_4,\e_4),\ldots ]\!]_o ,\llangle (a_1,\e_1),(a_0,\e_0),(a_{-1},\e_{-1}),\ldots \rrangle_o \big).
\end{split}
\end{equation*}

It is convenient to consider the extension $\tT_o$ of $\bT_o$ to $\tOmega_o :=\Omega_o \times \{ -1,1\}$ defined by
\begin{equation}\label{eq4}
\tT_o (x,y,\e) :=  \big( \bT_o (x,y), -\e_1(x) \e \big) ,
\end{equation}
with inverse
\begin{equation}\label{eq2.4}
\begin{split}
\tilde{T}^{-1}_o & \big([\![(a_1,\e_1),(a_2,\e_2),(a_3,\e_3),\ldots ]\!]_o ,
\llangle (a_0,\e_0),(a_{-1},\e_{-1}),(a_{-2},\e_{-2}),\ldots \rrangle_o ,\e\big)
\\ & = \big([\![(a_0,\e_0),(a_1,\e_1),(a_2,\e_2),\ldots ]\!]_o ,
\llangle (a_{-1},\e_{-1}),(a_{-2},\e_{-2}),(a_{-3},\e_{-3}),\ldots \rrangle_o ,-\e_0 \e\big).
\end{split}
\end{equation}

\subsection{Even continued fractions}\label{ecf}
The ECF expansion in $[0,1]\setminus\Q$ is given by
\begin{equation*}
x=[\![(a_1,\e_1), (a_2, \e_2), (a_3, \e_3), ...]\!]_e = \cfrac{1}{a_1+\cfrac{\e_1}{a_2+\cfrac{\e_2}{a_3+\cdots}}},
\end{equation*}
where $\e_i =\e_i(x)\in \{\pm1\}$ and $a_i =a_i(x)\in 2\N$.
Such an expansion is unique.
Every number $x\in [1,\infty) \setminus \Q$ has a unique infinite ECF expansion
\begin{equation*}
x=[\![(a_0,\e_0);(a_1,\e_1),(a_2,\e_2),\ldots ]\!]_e :=a_0+\e_0 [\![(a_1,\e_1),(a_2,\e_2),\ldots ]\!]_e \in [1,\infty),
\end{equation*}
with $\e_0\in \{ \pm 1\}$ and $a_0 \in 2\N$.

The even Gauss map $T_e$ acts on $[0,1]$ by
\begin{equation*}
T_e (x)= \e \bigg( \frac{1}{x}-2k\bigg)
\quad
\mbox{\rm if} \quad
x\in B(\e,k) =\begin{cases} \big( \frac{1}{2k+1},\frac{1}{2k}\big)
& \mbox{\rm if $\e=1$, $k\geqslant 1$} \\
\big( \frac{1}{2k},\frac{1}{2k-1}\big) & \mbox{\rm if $\e=-1$, $k\geqslant 1.$}
\end{cases}
\end{equation*}
Symbolically, $T_e$ acts on the ECF representation by removing the leading digit $(a_1,\e_1)$ of $x$, i.e.
\begin{equation*}
T_e \big( [\![(a_1,\e_1),(a_2,\e_2),(a_3,\e_3),\ldots]\!]_e \big) =[\![ (a_2,\e_2),(a_3,\e_3),(a_4,\e_4),\ldots ]\!]_e .
\end{equation*}
Here the infinite measure $d\mu_e (x)=\big( \frac{1}{1+x}+\frac{1}{1-x}\big) dx$
is $T_e$-invariant (see \cite{BL,Sch1,Sch2}).

The even continued fraction equivalent of the grotesque continued fractions are the extended even continued fractions.
Given $\e_i \in \{ \pm 1\}$ and even positive integers $b_i$, we denote
\begin{equation*}
y=\llangle (b_0,\e_0),(b_1,\e_1),\ldots \rrangle_e :=
\e_0 [\![(b_0,\e_1),(b_1,\e_2),(b_2,\e_3),\ldots ]\!]_e = \cfrac{\e_0}{b_0+\cfrac{\e_1}{b_1+\cdots}}
\in (-1,1).
\end{equation*}
The corresponding shift $\tau_e$ acts on $(-1,1)$ by
\begin{equation*}
\tau_e (y)=\frac{\e_0(y)}{y}-b_0(y),
\end{equation*}
or in the symbolic representation by
\begin{equation*}
\tau_e \big( \llangle (b_0,\e_0),(b_1,\e_1),(b_2,\e_2),\ldots \rrangle_e \big)
=\llangle (b_1,\e_1),(b_2,\e_2),(b_3,\e_3),\ldots \rrangle_e .
\end{equation*}

Consider $\Omega_e =(0,1)\times (-1,1)$ and
the natural extension of $T_e$, realized as the map \cite{Sch2}
\begin{equation*}
\bT_e :\Omega_e \rightarrow \Omega_e , \qquad \bT_e (x,y) =\bigg( T_e (x),\frac{\e_1(x)}{a_1(x)+y}\bigg).
\end{equation*}
Equivalently, $\bT_e$ is acting on $\Omega_e \cap (\R\setminus\Q)^2$ as a two-sided shift:
\begin{equation*}
\begin{split}
\bT_e &  \big([\![(a_1,\e_1),(a_2,\e_2),(a_3,\e_3),\ldots ]\!]_e ,\llangle (a_0,\e_0),(a_{-1},\e_{-1}),(a_{-2},\e_{-2}),\ldots \rrangle_e \big)
\\ & = \big([\![(a_2,\e_2),(a_3,\e_3),(a_4,\e_4)\ldots ]\!]_e ,\llangle (a_1,\e_1),(a_0,\e_0),(a_{-1},\e_{-1}),\ldots \rrangle_e \big).
\end{split}
\end{equation*}

It is also convenient to consider the extension $\tT_e$ of $\bT_e$ to $\tOmega_e =\Omega_e \times \{ -1,1\}$ defined by
\begin{equation*}
\tT_e (x,y,\epsilon) =  \big( \bT_e (x,y), -\e_1(x) \e \big) ,
\end{equation*}
with inverse
\begin{equation}\label{eq2.6}
\tT_e^{-1} (x,y,\e)= \big( \bT_e^{-1}(x,y), -\e_0(y) \e\big) .
\end{equation}

\section{Cutting sequences for geodesics on $\MM_o$}\label{oddseq}
\subsection{The group $\Gamma$ and the modular surface $\MM_o=\Gamma\backslash \H$}\label{oddsurface}
The fundamental Dirichlet region corresponding to the point $i$ and the
order three generators $S$ and $\ST$ is the quadrilateral
\begin{equation*}
{\mathfrak F} =\{ z\in \H : \lvert \Re z \rvert \leqslant \tfrac{1}{2} , \
\lvert z-1\rvert \geqslant 1,\  \lvert z+1\rvert \geqslant 1\} ,
\end{equation*}
with edges $[\omega,\infty]$, $[0,\omega]$ identified by $\ST =
\left( \begin{smallmatrix} 0 & 1 \\ -1 & 1 \end{smallmatrix}\right)$, and with
$[0,\omega^2]$, $[\omega^2,\infty]$ identified by $S$.
The resulting quotient space $\MM_o$ is homeomorphic to
the union of two cones glued along their basis $[0,\infty]$, with vertices
corresponding to $\omega$ and $\omega^2$, cuts along
the geodesic arcs $[\omega,\infty]$ and $[\omega^2,\infty]$,
and a cusp at $\pi_o (\infty)$ (see Figure \ref{Figure1}).

\begin{lemma}\label{lemma1}
The group $\Gamma$ is an index two subgroup of $\Gamma (1)$ and coincides with
\begin{equation*}
\Gamma_o  = \left\{ M \in \Gamma (1):
M\equiv  \left( \begin{matrix} 1 & 0 \\ 0 & 1 \end{matrix}\right) ,\
 \left( \begin{matrix} 0 & 1 \\ 1& 1 \end{matrix}\right) \
\mbox{or}\
\left( \begin{matrix} 1 & 1 \\ 1 & 0 \end{matrix} \right) \pmod{2} \right\} .
\end{equation*}
\end{lemma}

\begin{proof}
The generators $S$ and $T$ are contained in the group $\Gamma_o$. Since $\lvert \operatorname{SL}(2,\Z_2)\rvert =6$, we have
$[ \Gamma (1):\Gamma_o ]=2$. Finally, we notice that ${{\mathfrak F}={\mathcal F} \cup
\left( \begin{smallmatrix} 0 & -1 \\ 1 & 0 \end{smallmatrix}\right) {\mathcal F}}$, where
${\mathcal F}=\big\{ z\in \H : \lvert z\rvert \geqslant 1, \lvert
\operatorname{Re} z \rvert \leqslant \frac{1}{2} \}$ is the standard fundamental domain for
$\Gamma(1) \backslash \H$. This yields $[\Gamma (1):\Gamma ]=2$ and $\Gamma =\Gamma_o$.
\end{proof}

The lift of the group $\Gamma$ to $\operatorname{GL}(2,\Z)$ played an important role in the
analysis of the renewal time for odd continued fractions, pursued by Vandehey and one of the authors in \cite{BV}.

\begin{figure}
\centering\includegraphics[width=9cm]{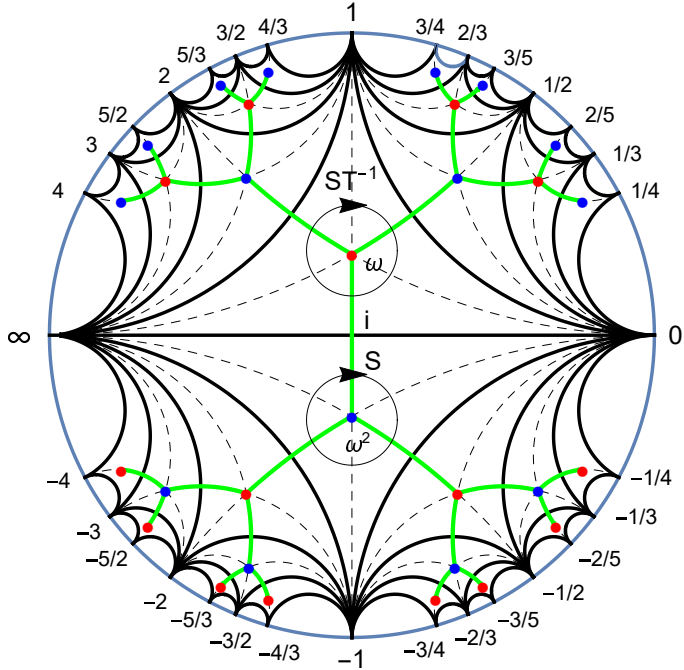}
\caption{The Farey tessellation in the disk model and the rotations $S$ and $ST^{-1}$}
\label{Figure3}
\end{figure}

Let $\AAA_o$ denote the set of geodesics $\gamma$ in $\H$ with endpoints satisfying
\begin{equation*}
(\gplus,\gneg) \in
\SSS_o :=\left( (1,\infty) \times (-I_G) \right) \cup \left( (-\infty,-1) \times I_G\right).
\end{equation*}

\begin{lemma}\label{Lemma2}
Every geodesic $\bar{\gamma}$ on $\MM_o$ lifts on $\H$ to a geodesic $\gamma \in \AAA_o$.
\end{lemma}

\begin{proof}
Without loss of generality, we can take $\bar{\gamma}$ to be a positively
oriented geodesic arc in ${\mathfrak F}$ connecting $[\omega^2,\infty]$ to $[\omega,\infty]$,
$[0,\omega^2]$ to $[\omega,\infty]$, $[\omega^2,\infty]$ to $[0,\omega]$, or
$[0,\omega^2]$ to $[0,\omega]$. Note that $[0,1]\subset I_G =(G-2,G)$. There are four cases to consider.
First, $\gamma_{-\infty} <-1<1<\gamma_\infty$, and an appropriate
$2\Z$-translation gives $\gamma_{-\infty} \in -I_G$ and $\gamma_\infty >1$.
Second, $\gamma$ connects the arc $[-1,0]$ to $[\omega,\infty]$, hence
$-1<\gamma_{-\infty} <0<1<\gamma_\infty$. Third, we have
$\gamma_{-\infty} <-1<0<\gamma_\infty<1$ so $\gamma\in\AAA_o$.

Finally, in the fourth case $TS^{-1}$ maps $\bar{\gamma}$ to a geodesic arc connecting
$[\omega+1,\infty]$ to $[\omega,\infty]$, hence $\gamma_{-\infty} >2>0>\gamma_\infty$.
An appropriate $2\Z$-translation ensures that $\gamma_{-\infty} \in I_G$ and $\gamma_\infty <-1$.
\end{proof}

\subsection{Cutting sequences and odd/grotesque continued fraction expansions}\label{oddcut}

As described in the introduction, our coding of geodesics on $\MM_o$ refines
the Series coding.  An oriented geodesic $\gamma$ in $\H$ is cut into segments as it crosses triangles in the Farey tessellation $\F$. Each segment of the geodesic crosses two sides of a triangle in the tessellation. If the vertex where the two sides meet is on the left, we label the segment $L$, if it is on the right we label it $R$.
We use $\rL$ and $\rR$ when the geodesic is in a light cell and $\bL$ or $\bR$ for a dark cell. This way, we assign to every geodesic
in $\H$ with irrational endpoints an infinite two-sided sequence of symbols
$\rL$, $\bL$, $\rR$ and $\bR$, with alternating shades.

Next, we analyze in detail the connection between the GCF expansion of
$\gamma_{-\infty}$, the OCF expansion of $\gamma_\infty$, and the strings 
in \eqref{string_types_o}.

For every M\"obius transformation $\bar{\rho}$ leaving
$\SSS_o$ invariant, we still denote by
$\bar{\rho}$ the product map $\bar{\rho}\times \bar{\rho}$,
viewed as a transformation of $\SSS_o$.
To every geodesic $\gamma \in \AAA_o$ we associate the positively
oriented geodesic arc $[\xi_\gamma,\eta_\gamma]$, where
\begin{equation*}
\xi_\gamma:=\begin{cases}
\gamma \cap [1,\infty] & \mbox{\rm if $\gamma_\infty >1$} \\
\gamma \cap [-1,\infty] & \mbox{\rm if $\gamma_\infty <-1$}
\end{cases}
\quad \mbox{\rm and} \quad
\eta_\gamma:=\begin{cases}
\gamma \cap [a_1,a_1+\e_1] & \mbox{\rm if $\gamma_\infty >1$} \\
\gamma \cap [-a_1,-a_1-\e_1] & \mbox{\rm if $\gamma_\infty <-1,$}
\end{cases}
\end{equation*}
with $(a_1,\e_1)=(a_1(\gamma_\infty),\e_1(\gamma_\infty))$.

When $\gplus>1$, we write $\gplus=[\![(a_1,\e_1);(a_2,\e_2),\ldots ]\!]_o$, $\gneg=-\llangle
(a_0,\e_0),(a_{-1},\e_{-1}), \ldots \rrangle_o \in -I_G$. When $\gplus<-1$, $\gplus=-[\![(a_1,\e_1);(a_2,\e_2),\ldots ]\!]_o$, $\gneg=\llangle
(a_0,\e_0),(a_{-1},\e_{-1}), \ldots \rrangle_o \in I_G$. Four cases will occur:

\begin{figure}[h]
\centering\includegraphics*[width=8cm]{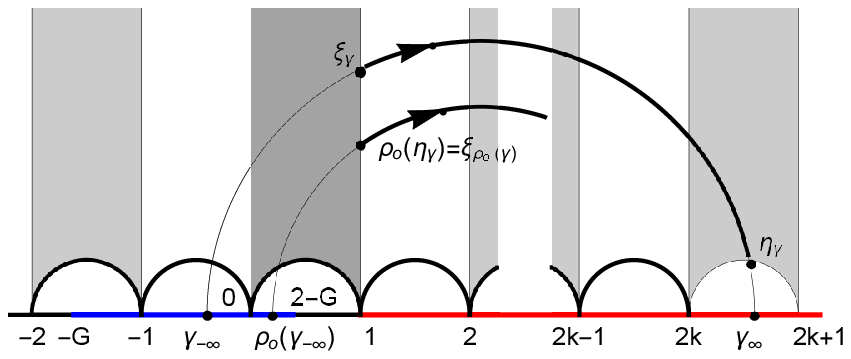}
\includegraphics*[width=7.5cm]{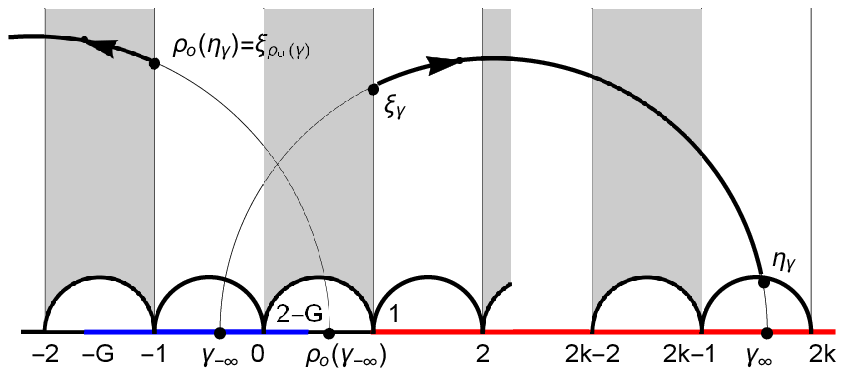}
\caption{The strings $\xi_\gamma (\rL \bL)^{k-1} \rL \bR \eta_\gamma$ and
$\xi_\gamma (\rL \bL)^{k-1} \rR \eta_\gamma$ }
\label{Figure4}
\end{figure}

(A) $\gamma_\infty \in(2k,2k+1)$, $\e_1=-1$, $a_1=2k+1$. The M\" obius transformation
$
\rho_o (x)=\frac{1}{a_1-x}
$
belongs to $\Gamma$ and
\begin{equation}\label{eq3.1}
\begin{split}
\rho_o  \big( [\![(a_1,\e_1); & (a_2,\e_2),(a_3,\e_3),\ldots ]\!]_o , -\llangle
(a_0,\e_0),(a_{-1},\e_{-1}),(a_{-2},\e_{-2}), \ldots \rrangle_o \big) \\ & =
\big( [\![(a_2,\e_2);(a_3,\e_3),(a_4,\e_4),,\ldots ]\!]_o ,
-\llangle (a_1,\e_1),(a_0,\e_0),(a_{-1},\e_{-1}),\ldots \rrangle_o \big) .
\end{split}
\end{equation}
In this situation $\rho_o$ transforms the arc $[\xi_\gamma,\eta_\gamma]$
of $\gamma$ connecting the geodesics
$[1,\infty]$ and $[a_1-1,a_1]$ into an arc connecting $[0,\frac{1}{a_1-1}]$ with
$[1,\infty]$. Following the orientation of $\gamma$, we assign the string $\xi_\gamma (\rL \bL)^{k-1}\rL \bR \eta_\gamma$ to the
arc $[\xi_\gamma,\eta_\gamma]$ (see Figure \ref{Figure4}).

(B) $\gamma_\infty \in(2k-1,2k),
\e_1=+1$, $a_1=2k-1$. The M\" obius transformation
$
\rho_o (x)=\frac{1}{a_1-x}
$
belongs to $\Gamma$ and
\begin{equation}\label{eq3.2}
\begin{split}
\rho_o  \big( [\![(a_1,\e_1); & (a_2,\e_2),(a_3,\e_3),\ldots ]\!]_o , -\llangle
(a_0,\e_0),(a_{-1},\e_{-1}),(a_{-2},\e_{-2}), \ldots \rrangle_o \big) \\ & =
\big( -[\![(a_2,\e_2);(a_3,\e_3),(a_4,\e_4),\ldots ]\!]_o ,
\llangle (a_1,\e_1),(a_0,\e_0),(a_{-1},\e_{-1}),\ldots \rrangle_o \big) .
\end{split}
\end{equation}
In this situation $\rho_o$ transforms the arc $[\xi_\gamma,\eta_\gamma]$
of $\gamma$ connecting the geodesics $[1,\infty]$ and $[a_1,a_1+1]$ into an arc connecting $[0,\frac{1}{a_1-1}]$ with
$[-1,\infty]$. Following the orientation of $\gamma$, to the
arc $[\xi_\gamma,\eta_\gamma]$ we assign the string $\xi_\gamma (\rL \bL)^{k-1} \rR \eta_\gamma$
(see Figure \ref{Figure4}).

(C) $\gamma_\infty \in(-2k-1,-2k), \e_1=-1$, $a_1=2k+1$. The M\" obius transformation
$
\rho_o (x)=\frac{1}{-a_1-x}
$
belongs to $\Gamma$ and
\begin{equation}\label{eq3.3}
\begin{split}
\rho_o  \big( -[\![(a_1,\e_1); & (a_2,\e_2),(a_3,\e_3),\ldots ]\!]_o , \llangle
(a_0,\e_0),(a_{-1},\e_{-1}),(a_{-2},\e_{-2}), \ldots \rrangle_o \big) \\ & =
\big( -[\![(a_2,\e_2);(a_3,\e_3),(a_4,\e_4),\ldots ]\!]_o ,
\llangle (a_1,\e_1),(a_0,\e_0),(a_{-1},\e_{-1}),\ldots \rrangle_o \big) .
\end{split}
\end{equation}
In this situation $\rho_o$ transforms the arc $[\xi_\gamma,\eta_\gamma]$
of the geodesic $\gamma$ connecting the geodesics
$[-1,\infty]$ and $[-a_1,-a_1+1]$ into an arc connecting $[\frac{-1}{a_1-1},0]$ with
$[-1,\infty]$. Following the orientation of $\gamma$, to the
arc $[\xi_\gamma,\eta_\gamma]$ we assign the string $\xi_\gamma (\bR \rR)^{k-1} \bR \rL\eta_\gamma$
(see Figure \ref{Figure6}).

\begin{figure}
\centering\includegraphics*[width=8cm]{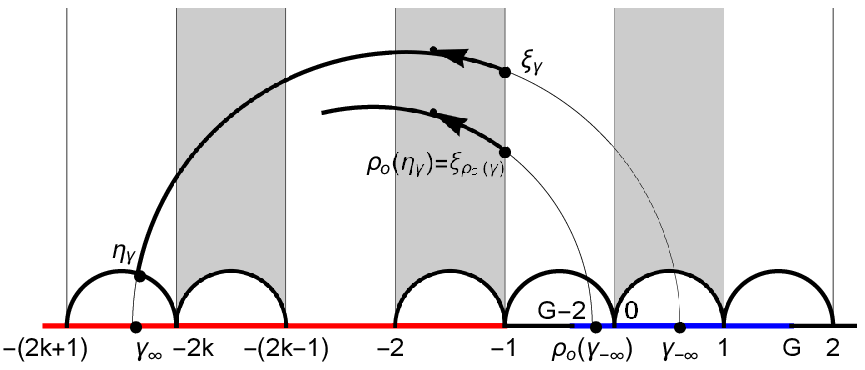}
\centering\includegraphics*[width=7.7cm]{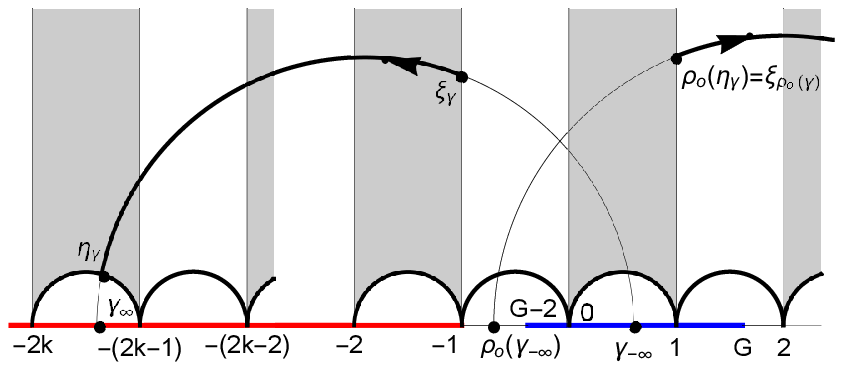}
\caption{The strings $\xi_\gamma (\bR \rR)^{k-1} \bR \rL\eta_\gamma$ and
$\xi_\gamma (\bR \rR)^{k-1} \bL \eta_\gamma$ }
\label{Figure6}
\end{figure}

(D) $\gamma_\infty \in (-2k,-2k+1), \e_1=+1$, $a_1=2k-1$. The M\" obius transformation
$
\rho_o (x)=\frac{1}{-a_1-x}
$
belongs to $\Gamma$ and
\begin{equation}\label{eq3.4}
\begin{split}
\rho_o  \big( -[\![(a_1,\e_1); & (a_2,\e_2),(a_3,\e_3),\ldots ]\!]_o , \llangle
(a_0,\e_0),(a_{-1},\e_{-1}),(a_{-2},\e_{-2}), \ldots \rrangle_o \big) \\ & =
\big( [\![(a_2,\e_2);(a_3,\e_3),(a_4,\e_4),\ldots ]\!]_o ,
-\llangle (a_1,\e_1),(a_0,\e_0),(a_{-1},\e_{-1}),\ldots \rrangle_o \big) .
\end{split}
\end{equation}
In this situation $\rho_o$ transforms the arc $[\xi_\gamma,\eta_\gamma]$
of the geodesic $\gamma$ connecting the geodesics
$[-1,\infty]$ and $[-a_1-1,-a_1]$ into an arc connecting $[\frac{-1}{a_1-1},0]$ with
$[1,\infty]$. Following the orientation of $\gamma$, to the
arc $[\xi_\gamma,\eta_\gamma]$ we assign the string $\xi_\gamma (\bR \rR)^{k-1} \bL \eta_\gamma$
(see Figure \ref{Figure6}).

Summarizing \eqref{eq3.1}-\eqref{eq3.4}, we obtain the following general formula for the action of $\rho_o$ on $\SSS_o$,
where $\epsilon =+1$ in cases A and B, and $\epsilon =-1$ in cases C and D:
\begin{equation}\label{eq3.5}
\begin{split}
\rho_o & \big( \e [\![ (a_1,\e_1);(a_2,\e_2),(a_3,\e_3),\ldots ]\!]_o, -\e
\llangle (a_0,\e_0),(a_{-1},\e_{-1}),(a_{-2},\e_{-2}),\ldots \rrangle_o) \\
& = (-\e_1) \big( \e [\![ (a_2,\e_2);(a_3,\e_3),(a_4,\e_4),\ldots ]\!]_o,
-\e \llangle (a_1,\e_1),(a_0,\e_0),(a_{-1},\e_{-1}),\ldots \rrangle_o \big),
\end{split}
\end{equation}
with inverse
\begin{equation*}
\begin{split}
\rho_o^{-1} & \big( \e [\![ (a_1,\e_1);(a_2,\e_2),(a_3,\e_3),\ldots ]\!]_o, -\e
\llangle (a_0,\e_0),(a_{-1},\e_{-1}),(a_{-2},\e_{-2}),\ldots \rrangle_o) \\
& = (-\e_0) \big( \e [\![ (a_0,\e_0);(a_1,\e_1),(a_2,\e_2),\ldots ]\!]_o,
-\e \llangle (a_{-1},\e_{-1}),(a_{-2},\e_{-2}),(a_{-3},\e_{-3}),\ldots \rrangle_o \big).
\end{split}
\end{equation*}
This also gives that $\rho_o$ reflects across the imaginary axis exactly when $\e_1=+1$, so that $\rho_o$ agrees with
the transformation $\rho$ from \cite{Ser} exactly where the RCF and OCF agree.
We also get that cases A and C are followed by A or B, and cases B and D are followed by cases C or D.

\begin{proposition}\label{Lcommdiagodd}
The map $\rho_o :\SSS_o \rightarrow \SSS_o$ is invertible, and the diagram
$$
\xymatrix{ \mbox{$\SSS_o$} \ar@{->} [r]^(.5){\rho_o}
\ar@{->} [d]_(.4) {J_o}  & \mbox{$\SSS_o$} \ar@{->} [d]^(.4) {J_o} \\
\mbox{$\tOmega_o$} \ar@{->} [r] ^(.5){\tT_o} & \tOmega_o } $$
commutes, where $J_o:\SSS_o \rightarrow \tOmega_o$ is the invertible map defined by
\begin{equation*}
J_o (x,y):=\operatorname{sign} x (1/x,-y,1) = \begin{cases}
(1/x,-y,1) & \mbox{\rm if $x> 1$, $y\in -I_G$} \\
(-1/x,y,-1) & \mbox{\rm if $x< -1$, $y\in I_G.$}
\end{cases}
\end{equation*}
\end{proposition}

\begin{proof}
Set $x= \e [\![ (a_1,\e_1); (a_2,\e_2),\ldots ]\!]_o$,
$y=-\e \llangle (a_0,\e_0),(a_{-1},\e_{-1}),\ldots \rrangle_o$ with
$\e =\pm 1$, so that $(x,y)\in \SSS_o$.
Using formulas \eqref{eq3.5} and \eqref{eq2.4}, we get that
\begin{equation*}
\begin{split}
J_o \rho_o (x,y) & = J_o \big( -\e_1 \e [\![(a_2,\e_2); (a_3,\e_3),\ldots ]\!]_o,
\e_1 \e \llangle  (a_1,\e_1),(a_0,\e_0),\ldots \rrangle_o \big) \\
& = \big( [\![(a_2,\e_2),(a_3,\e_3),\ldots ]\!]_o, \llangle (a_1,\e_1),(a_0,\e_0),\ldots \rrangle_o ,-\e_0 \e  \big) \\
& = \tT_o \big( [\![(a_1,\e_1),(a_2,\e_2),\ldots ]\!]_o,
\llangle (a_0,\e_0),(a_{-1},\e_{-1}),\ldots \rrangle_o ,\e \big)
= \tT_o J_o (x,y). \qedhere
\end{split}
\end{equation*}
\end{proof}

\begin{cor}\label{cor4}
If $\alpha =[\![\overline{(a_1,\e_1);(a_2,\e_2), \ldots,(a_r,\e_r)}]\!]_o>1$ and
$\beta=-\llangle \overline{((a_r,\e_r),\ldots,(a_1,\e_1)} \rrangle_o\in -I_G$, then
\begin{itemize}
\item[(i)] $\  \rho_o^r (\alpha,\beta)=(-\e_1)\cdots (-\e_r) (\alpha,\beta)$.
\item[(ii)] $\  \rho_o^{2r} (\alpha,\beta) =(\alpha,\beta)$.
\end{itemize}
\end{cor}

\begin{proof}
To check (i), we use Proposition \ref{Lcommdiagodd} and compute
\begin{equation*}
\begin{split}
J_0^{-1} \tT_o^{r} J_0 (\alpha,\beta) & =J_0^{-1} \tT_o^{r} (1/\alpha,-\beta,1) =
J_0^{-1} \big( 1/\alpha,-\beta,(-\e_1)\cdots (-\e_r)\big)  =(-\e_1)\cdots (-\e_r) (\alpha,\beta) .
\end{split}
\end{equation*}

(ii) We consider the case $(-\e_1)\cdots (-\e_r)=-1$,
when we use Proposition \ref{Lcommdiagodd} and
\begin{equation*}
\begin{split}
J_o^{-1} \tT_o^{r} J_o (-\alpha,-\beta) & =J_o^{-1} \tT_o^{r} (1/\alpha,-\beta,-1)
=J_o^{-1} \big( 1/\alpha,-\beta,-(-\e_1)\cdots (-\e_r) \big) \\
& =J_o^{-1} (1/\alpha,-\beta,1) =(\alpha,\beta). \qedhere
\end{split}
\end{equation*}
\end{proof}

Finally, we notice the equality
\begin{equation}\label{3.6}
\bT_o^{-1} (u,v)=\bigg( -\frac{\operatorname{sign}(v)}{\rho_o (1/u)},\operatorname{sign} (v) \rho_o (-v)\bigg),
\quad \forall (u,v)\in\Omega_o \cap (\R\setminus \Q)^2 .
\end{equation}

\section{Connection with cutting sequence and RCF}\label{cut}
\subsection{Odd continued fractions}\label{cutodd}
We now explore the connection between the cutting sequences of the regular continued fractions and the odd continued fractions. Here, we
use $x$ to mark the imaginary axis, as in \cite{Ser}, and $\xi_{\gamma}=\gamma\cap \pm [1,\infty]$ as defined above.

In cases $A$ and $B$, we get the cutting sequence $\ldots xL^{n_1}R^{n_2}L^{n_3}\ldots$ with regular continued fraction expansion $[n_1;n_2,n_3,\ldots]$. Without coloring, this corresponds to $\ldots L \x L^{n_1-1}R^{n_2}L^{n_3}\eta_\gamma\ldots$. We have two cases to consider for the first digit of the odd continued fraction expansion.

\begin{description}

\item[(A)] $n_1=2k$ is even, and $\gplus\in(2k,2k+1)$. This gives the cutting sequence $\ldots \x (\rL \bL)^{k-1}\rL \bR\eta_\gamma$, and we get $(2k+1,-1)=(a_1,\e_1)$. The next digit is represented by $L^n R$ for $n\geqslant 0$.
When $n_2>1$, the next digit is $L^0\R$, corresponding to $(1,+1)$. This gives the cutting sequence $\ldots \x(\rL \bL)^{k-1}\rL \bR \eta_\gamma\rR\ldots$. This corresponds to \[2k+\cfrac{1}{n_2+z}=2k+1-\cfrac{1}{1+\cfrac{1}{n_2-1+z}}.\]
When $n_2=1$, we proceed with $\ldots \x (\rL \bL)^{k-1}\rL \bR\eta_\gamma \rL\ldots$. This corresponds to
\[2k+\cfrac{1}{1+\frac{1}{n_3+z}}=2k+1+\frac{-1}{n_3+1+z}.\]
These equalities correspond to the singularization and insertion algorithm
introduced by Kraaikamp in \cite{Kra} and explicitly computed for the OCF in
Masarotto's master's thesis \cite{Mas} (see also \cite{HK}).
This algorithm is based on the identity
\[a+\cfrac{\e}{1+\frac{1}{b+z}}=a+\e+\frac{-\e}{b+1+z}
\quad \mbox{\rm where $\epsilon \in \{ \pm 1\}$.}
\]

\item[(B)] $n_1=2k-1$ is odd, and $\gplus\in(2k-1,2k)$. This gives the cutting sequence $\ldots \x (\rL \bL)^{k-1}\rR\eta_\gamma$, and we get $(2k-1,+1)=(a_1,\e_1)$. The next digit is represented by $R^nL$, where $n\geqslant 0$.  If $n_2=1$, we have $ \ldots \x (\rL \bL)^{k-1}\R\eta_\gamma \bL \ldots$, and the next digit corresponds to $R^0\bL$, which gives $(1,+1)$.
\end{description}
Strings starting with $\bR$ are treating similarly, with $(\bR\rR)^{k-1}\bL$ corresponding to $(2k-1, +1)$ and $(\bR\rR)^{k-1}\bR\rL$  to $(2k+1, -1)$.

Geodesics of Type C and D can be classified similarly. In this case, we get the cutting sequence $\ldots xR^{n_1}L^{n_2}R^{n_3}\ldots$, where $\gplus=-[n_1;n_2,n_3,\ldots]$. This gives, without coloring, the cutting sequence  $\ldots R \x R^{n_1-1}L^{n_2}R^{n_3}\ldots$, and we interpret the strings in the same way as above.

\subsection{Grotesque continued fractions}\label{grot}
The grotesque continued fractions are the dual continued fraction expansion of
the odd continued fractions, which changes the restriction on the digits.
That is, for
$\e_1 /(a_1+ \e_2 /\ldots)$,
the odd continued fractions require $a_i+\e_{i+1}\geqslant 2$, and the grotesque require $a_i+\e_{i}\geqslant 2$.
This means we need a different insertion and singularization algorithm to convert regular continued fractions to
grotesque continued fractions that the one used for odd continued fractions. As with Series' description of the regular
continued fractions \cite{Ser}, the forward endpoints are read from left to right, but the backwards endpoints are
read from right to left. Thus, for the odd continued fractions, we can read the strings one at a time, but for
grotesque continued fractions, we must also consider whether the preceding string would be valid.

We consider cases A, B, C, and D similar to those above. To stay consistent with how we normally read, we say that
the string ends in the letter on the right. The preceding string is the string to the left of the one we are
considering.

\begin{description}
\item[(A)] $\gplus>1, \gneg\in(0,2-G)$, and $\e_0=-1$.
We get the cutting sequence  $\ldots (\rL\bL)^{k-1}\rL\bR \x\ldots$ and $(a_0,\e_0)=(2k+1,-1)$ as in Figure \ref{Figure4}.
The preceding string must end in $\bR$ or $\bL$, as in case A or B.

\item[(B)] $\gplus>1, \gneg\in(-G,0),$ and $\e_0=+1$.
We get the cutting sequence $\ldots (\bR\rR)^{k-1}\bL\x\ldots$ and $(a_0,\e_0)=(2k-1,+1)$, as in Figure \ref{Figure4}.
Note that we cannot have $a_0=1$.
The preceding string must end in $\R$ or $\rL$, as in case C or D.

\item[(C)] $\gplus<-1,\gneg\in(G-2,0),$ and $\e_0=-1$.
We get the cutting sequence $\ldots (\bR\rR)^{k-1}\bR\rL\x\ldots$ and $(a_0,\e_0)=(2k+1,-1)$
as in the image on the left of Figure \ref{Figure6}.
The preceding string must end in $\R$ or $\rL$, as in case C or D.

\item[(D)] $\gplus<-1, \gneg\in(0, G),$ and $\e_0=+1$.
We get the cutting sequence $\ldots (\rL\bL)^{k-1}\rR \x\ldots$ and $(a_0,\e_0)=(2k-1,+1)$
as in the image on the right of Figure \ref{Figure6}.
The preceding string must end in $\bR$ or $\bL$, as in case A or B.

\end{description}

To illustrate the difference between odd and grotesque continued fraction expansions, we consider two example strings.
First, note that for case C, we have the RCF cutting sequence $\ldots L^{n_{-2}} R^{n_{-1}}L^{n_0}x R\ldots$ for the RCF
$[n_0,n_{-1},n_{-2},\ldots]$, where $x=\gamma\cap [0,\infty]$. If we ignore coloring, this corresponds to the string $\ldots L^{n_{-2}} R^{n_{-1}}L^{n_0-1}R \x\ldots$.
We compare the GCF expansions for $[2,1,2,\ldots]$ and $[2,2,2,\ldots]$, as well as the OCF expansion of $[2;2,2,\ldots]$.

In the first case, we have the string $\ldots\rR\bL\rL\bR\rL\bL\rR \x\ldots$. Following the above rules, we get the grouping $\ldots\rR\bL(\rL\bR)(\rL\bL\rR)\x\ldots$, and
$[2,1,2,\ldots]=\llangle (3,+1),(3,-1),\ldots\rrangle_o$.

In the second case, we have $\ldots\bR\rL\bL\rR\bR\rL\bL\rR\x\ldots$, and $[2,2,2,\ldots]=\llangle (1,+1),(1,+1),(3,-1),\ldots\rrangle_o$. We need to change the grouping for the first digit to make the second digit an allowable string, since $\bR$ must be preceded by $(\rL\bL)^{k-1}\rL$. Thus, we get
$\dots\bR\rL\bL\rR(\bR\rL)(\bL)(\rR)\x\dots$, where the final grouping is determined based on whether the previous letter is $\rL$ or $\rR$.

Finally, we contrast this with the regular continued fraction $\frac{1}{\gneg}=[2;2,2,\ldots]$. The corresponding cutting sequence $\ldots\x (\rL\bR)(\rR)( \bL )\rL\bR \ldots$ gives the OCF. We can see the connection between the RCF, GCF, and OCF by noting that
\[2+\cfrac{1}{2+z}=1+\cfrac{1}{1-\cfrac{1}{3+z}}=3-\cfrac{1}{1+\cfrac{1}{1+z}}.
\]

\section{Applications}
\subsection{Invariant measures}\label{invmeas}
Section \ref{oddseq} provides  a cross-section $\XX$ for the geodesic flow on $T_1 (\MM_o)$
consisting of those elements $(\xi_\gamma,u_\gamma)$ with base point $\xi_\gamma$ on
$\pi_o (\pm 1 +i\R)$ such that the cutting sequence of $\pi_o(\gamma)$ is broken
into segments as in \eqref{string_types_o}, and the unit vector $u_\gamma$ points along the geodesic.
Proposition \ref{Lcommdiagodd} shows that the first return map
of the geodesic flow to $\XX$ corresponds to the transformation $\tT_o$ on $\tilde{\Omega}_o$.

As in \cite{Ser}, it is convenient to express $(x,u)\in T_1(\H)$ in coordinates
$(\alpha,\beta,t)$ given by the endpoints $\alpha,\beta$ of the geodesic $\gamma(u)$
through $u$ and the distance $t$ from the midpoint of $\gamma(u)$ to $u$. It is possible to transform
the usual measure Haar measure $y^{-2} dx dy d\theta$ on $\H \times S^1 \cong T_1 (\H) \cong \operatorname{PSL}(2,\R)$
to $(\alpha-\beta)^{-2} d\alpha d\beta dt$.
Section 3 shows that $\SSS_o$ and $\XX$ are naturally identified, up to a null-set,
by mapping
$\gamma\in\SSS_o$ with irrational endpoints $(\gamma_\infty,\gamma_{-\infty})\in \SSS_o$ to
$(\xi_\gamma,u_\gamma)\in \XX_o$ as above.
The first return map to $\XX$ is subsequently identified with $\rho_o$, acting on $\SSS_o$
with corresponding invariant measure $\mu=(\alpha-\beta)^{-2} d\alpha d\beta$.
We define $\pi_o,\pi_1,\pi_2$ on $\tilde{\Omega}_o$ by $\pi_o (x,y,\epsilon)=(x,y), \pi_1 (x,y,\epsilon)=x$, and $\pi_2 (x,y,\epsilon)=y$.
The push-forward of $\mu$ under
$\pi_o\circ J_o$ provides a $\bar{T}_o$-invariant measure $\bar{\mu}_o$ on $\Omega_o$. The push-forward of $\mu$ under
$\pi_1 \circ J_o$ provides a $T_o$-invariant measure $\mu_o$ on $(0,1)$. Finally, the push-forward of $\mu$ under
$\pi_2 \circ J_o$
 provides a $\tau_o$-invariant
measure $\nu_o$ on $I_G$.
For every rectangle $E=[a,b]\times [c,d] \subset \Omega_o$, we have
\[(\pi_o \circ J_o)^{-1}(E) =[b^{-1},a^{-1}] \times [-d,-c] \cup
[-a^{-1},-b^{-1}] \times [c,d] \qquad \mbox{\rm and} \]
\begin{equation*}
\bar{\mu}_o (E) =
\int_{1/b}^{1/a} \int_{-d}^{-c} \frac{d\alpha\, d\beta}{(\alpha-\beta)^2} +
\int_{-1/a}^{-1/b} \int_c^d \frac{d\alpha\, d\beta}{(\alpha-\beta)^2}
= 2\iint_E \frac{dx\, dy}{(1+xy)^2} ,
\end{equation*}
showing that $\bar{\mu}_o =(1+xy)^{-2}dx dy$ is a finite
$\bar{T}_o$-invariant measure.
Using $\frac{1}{G}=G-1$ and $\frac{1}{2-G}=G+1$,
we see that for every interval $[a,b]\subset (0,1)$,
\begin{equation*}
\begin{split}
\mu_o ([a,b]) & = \bar{\mu}_o \big( [1/b,1/a] \times (-I_G) \cup [-1/a,-1/b] \times I_G \big) \\
& = \int_{1/b}^{1/a} \int_{-G}^{2-G} \frac{d\alpha\, d\beta}{(\alpha-\beta)^2} +
\int_{-1/a}^{-1/b} \int_{G-2}^G \frac{d\alpha\, d\beta}{(\alpha-\beta)^2} \\
& = \int_a^b \int_{-G}^{2-G} \frac{1}{u^2}\cdot \frac{du\, dv}{(1/u-v)^2} +
\int_a^b \int_{G-2}^G \frac{1}{u^2} \cdot \frac{du\, dv}{(-1/u-v)^2} \\
& =2 \int_a^b \left( \frac{1}{u+G-1}-\frac{1}{u-G-1}\right) du.
\end{split}
\end{equation*}
This gives that $\mu_o =( \frac{1}{u+G-1}-\frac{1}{u-G-1}) du$ is a finite
$T_o$-invariant measure.
Finally, for every interval $[c,d] \subset I_G$ we have
\begin{equation*}
\begin{split}
\nu_o ([c,d]) & = \bar{\mu}_o \big( [1,\infty) \times [-d,-c] \cup (-\infty,-1] \times [c,d]\big)
\\ & = \int_1^\infty \int_{-d}^{-c} \frac{d\alpha\, d\beta}{(\alpha-\beta)^2} +
\int_{-\infty}^{-1} \int_c^d \frac{d\alpha\, d\beta}{(\alpha-\beta)^2} \\ &
 = 2 \int_0^1 \int_c^d \frac{1}{u^2}\cdot \frac{du\, dv}{(1/u+v)^2} =
2\int_c^d \frac{dv}{1+v} ,
\end{split}
\end{equation*}
showing that $\nu_o=\frac{dv}{1+v}$ is a finite $\tau_o$-invariant measure.

\subsection{Quadratic surds and their conjugates}
\begin{proposition}\label{cor5}
A real number $\alpha >1$ has a purely periodic OCF expansion if and only if
$\alpha$ is a quadratic surd with $-G < \balpha <2-G$. Furthermore, if
\begin{equation}\label{eq5.1}
\alpha =[\![ \, \overline{(a_1,\e_1); (a_2,\e_2), \ldots ,(a_{r},\e_{r})} \,]\!]_o ,
\end{equation}
then
\begin{equation}\label{eq5.2}
\balpha=-\llangle \,\overline{(a_{r},\e_{r}), \ldots, (a_1,\e_1)} \,\rrangle_o .
\end{equation}
\end{proposition}

\begin{proof}
In one direction, suppose that $\alpha$ is given by \eqref{eq5.1}. Consider
the geodesic $\gamma \in \AAA_o$ with endpoints at $\gamma_\infty=\alpha$ and
$\gamma_{-\infty}=\beta =- \llangle \,\overline{(a_{r},\e_{r}), \ldots, (a_1,\e_1)} \,\rrangle_o
\in -I_G$. Corollary \ref{cor4} shows that the geodesic $\gamma$ is fixed by
$\rho_o^{2r}$, so it
is fixed by some $M\in\Gamma$, $M\neq I$. Hence both $\alpha$ and $\beta$ are fixed by
$M$; in particular, $\beta=\balpha$.

In the opposite direction, suppose that $A\alpha^2+B\alpha+C=0$ with
$\operatorname{gcd} (A,B,C)=1$, $A\geqslant 1$, and $\balpha \in -I_G$.
The quadratic surds $\alpha$, $\balpha$, $-\alpha$, $\overline{-\alpha}=-\balpha$,
and $M\alpha=\frac{a\alpha+b}{c\alpha+d}$ with $M\in \Gamma \subset \Gamma(1)$
have the same discriminant.
From $\alpha-\balpha = \frac{\sqrt{\Delta}}{A} > G-1$ and $-2AG < 2A\balpha =
-B-\sqrt{\Delta} < 2A(2-G)$, we infer that the number of quadratic surds $\alpha$
with fixed discriminant $\Delta=B^2-4AC$ must be finite.
Employing equality \eqref{3.6}, it follows that both components of
$\bT_o^k (\frac{1}{\alpha},-\balpha)$ are quadratic surds with
discriminant $\Delta$ for every $k\geqslant 0$. Since they satisfy the same kind of restrictions as $\alpha$ above,
there exist $k,k^\prime \geqslant 0$,
$k\neq k^\prime$ such that $\bT^k_o (\frac{1}{\alpha},-\balpha)=\bT^{k^\prime}_o (\frac{1}{\alpha},-\balpha)$.
The map $\bT_o$ is invertible, hence there exists $r\geqslant 1$ such that
$\bT^r_o (\frac{1}{\alpha},-\balpha)=(\frac{1}{\alpha},-\balpha)$, showing that
$\alpha$ must be of the form \eqref{eq5.1} and $\balpha$ of the form \eqref{eq5.2}.
\end{proof}

\subsection{Action of $\Gamma$ on $\R$ and continued fractions}
Define the $m$-tail of an irrational number $\alpha =[\![ (a_1,\e_1);(a_2,\e_2),\ldots ]\!]_o >1$ by
\begin{equation*}
t_m (\alpha):= (-\e_1)\cdots (-\e_m) [\![ (a_{m+1},\e_{m+1}), (a_{m+2},\e_{m+2}), \ldots ]\!]_o .
\end{equation*}

\begin{proposition}\label{cor6}
Two irrationals $\alpha,\beta  >1$ are $\Gamma$-equivalent if and only if there exist
$r,s\geqslant 0$ such that
\begin{equation}\label{eq5.3}
t_r(\alpha)=t_s (\beta) .
\end{equation}
\end{proposition}

\begin{proof}
The proof follows closely the outline of statement 3.3.3 in \cite{Ser}.
In one direction, if \eqref{eq5.3} holds, then $\alpha$ and $\beta$ are
$\Gamma$-equivalent because $-a_1-\frac{1}{\alpha} =t_1(\alpha)$.

Conversely, suppose that $g\alpha=\beta$ for some $g\in\Gamma$.
Fix $\delta \in -I_G$ and consider the geodesics $\gamma,\gamma^\prime \in {\mathcal A}_o$ with
$\gamma_{-\infty} =\gamma_{-\infty}^\prime =\delta$, $\gamma_\infty=\alpha$ and
$\gamma^\prime_\infty=\beta$. Their cutting sequences are
$\ldots \xi_\gamma A_1 A_2 \ldots$ and  $\ldots\xi_\gamma B_1 B_2 \ldots$, respectively,
with $A_i,B_i$ strings of type $A$, $B$, $C$ or $D$. The geodesics
$\gamma^{\prime\prime} =g\gamma$ and $\gamma^\prime$ have the same endpoint $\beta$.
Since their $\Gamma(1)$-cutting sequences in $L$ and $R$ coincide
(cf. \cite[Lemma 3.3.1]{Ser}), their cutting sequences also coincide, implying that
the cutting sequence of $\gamma^{\prime\prime}$ is of the form
$\xi_{\gamma^{\prime\prime}} \ldots B_k B_{k+1}\ldots$ for some $k\geqslant 1$.
As $\gamma$ and $\gamma^{\prime\prime}$ are $\Gamma$-equivalent geodesics, their cutting sequences
(after equivalent initial points) will coincide, implying that the cutting sequences of
$\gamma$ and $\gamma^\prime$ are of the form $\ldots \xi_\gamma A_1 \ldots A_r D_1 D_2 \ldots$
and $\ldots \xi_\gamma B_1 \ldots B_s D_1 D_2 \ldots$ respectively.
Upon \eqref{eq3.5} this implies \eqref{eq5.3}.
\end{proof}

Two rational numbers are $\Gamma$-equivalent, as shown by the following elementary

\begin{lemma}\label{lemma7}
$\ \Gamma \infty =\Q$.
\end{lemma}

The problem of characterizing $\Gamma$-equivalence classes for a
broad class of subgroups of $\Gamma(1)$ has been recently investigated,
with a different approach, in \cite{Pan}.

\begin{proposition}\label{cor8}
The OCF expansion of an irrational $\alpha$ is eventually periodic if and only if $\alpha$ is a
quadratic surd.
\end{proposition}

\begin{proof}
If the OCF tail of $\alpha$ is eventually periodic, then
$g\alpha =\e [\![ \overline{(a_1,\e_1),\ldots ,(a_r,\e_r)} ]\!]_o$ for some $g\in\Gamma$ and
$\e =\pm 1$. Proposition \ref{cor5} gives that $g\alpha$ is a quadratic surd, hence $\alpha$ is a quadratic surd.

Conversely, assume that $\alpha$ is a quadratic surd and $\gamma$ is the geodesic connecting $\alpha$
to its conjugate root $\overline{\alpha}$. Let $g\in\Gamma$ such that the geodesic $g\gamma$ is a lift of $\pi_o (\gamma)$ to
$\H$ with $g\gamma\in {\mathcal A}_o$. We can assume that $g\alpha >1$, reversing $\alpha$ and $\overline{\alpha}$ if necessary.
Proposition \ref{cor5} shows that $g\alpha$ has purely periodic OCF expansion, and
Proposition \ref{cor6} shows that the OCF expansion of $\alpha$ is eventually periodic.
\end{proof}

Proposition \ref{cor8} is known in more general situations, for instance it holds for all
$D$-continued fractions (see, e.g., \cite{Mas}).

\subsection{Closed geodesics on $\MM_o$}
Employing Proposition \ref{cor6} and standard arguments, one can prove
\begin{proposition}\label{cor9}
A geodesic $\bar{\gamma}$ on $\MM_0$ is closed if and only if
it has a lift $\gamma\in \AAA_o$ with purely
periodic endpoints
\begin{equation*}
\gamma_\infty =\e [\! [ \overline{(a_1\,e_1);(a_2,\e_2),\ldots, (a_r,\e_r)} ]\!]_o \quad \mbox{and}
\quad \gamma_{-\infty} =-\e \llangle \overline{(a_r,\e_r),\ldots ,(a_2,\e_2),(a_1,\e_1)} \rrangle_o
\end{equation*}
for some $\e=\pm 1$ and $(-\e_1)\cdots (-\e_r)=1$.
\end{proposition}

\subsection{The roof function and length of closed geodesics on $\MM_o$}
Our construction describes the geodesic flow on $T_1 (\MM_o)$ as a suspension flow over
the measure preserving transformation $(\tOmega_o, \tT_o,\tilde{\mu}_o)$ identified with
$(\SSS_o, \rho_o,(\alpha-\beta)^{-2} d\alpha d\beta)$.
The roof function is given by the hyperbolic distance between two consecutive return points to $\SSS_o$:
\begin{equation*}
r_o (\xi_\gamma )=d (\xi_\gamma,\eta_\gamma),
\end{equation*}
with $\xi_\gamma,\eta_\gamma\in \H$ as in Subsection \ref{oddcut}. The points $\xi_\gamma$ and $\eta_\gamma$
can also be identified with elements in $T_1 (\MM_o)$ and represent two consecutive
changes in type for the cutting sequence of the geodesic $\gamma$.
It is convenient to replace the geodesic arc $[\xi_\gamma,\eta_\gamma]$
with $[\xi,\eta]$, where $\xi=\rho_o(\xi_\gamma)$, $\eta=\rho_o (\eta_\gamma)$. Considering
$\gamma_- =\rho(\gamma_{-\infty})$, $\gamma_+=\rho_o(\gamma_\infty)$, we employ as in \cite{Ser} the formula
\begin{equation*}
d(\xi_\gamma,\eta_\gamma)= d(\xi,\eta) =\log \bigg|
\frac{\gamma_{-}-\eta}{\gamma_{-}-\xi} \cdot
\frac{\gamma_{+}-\xi}{\gamma_{+}-\eta} \bigg| .
\end{equation*}

\begin{figure}
\centering\includegraphics[width=8cm]{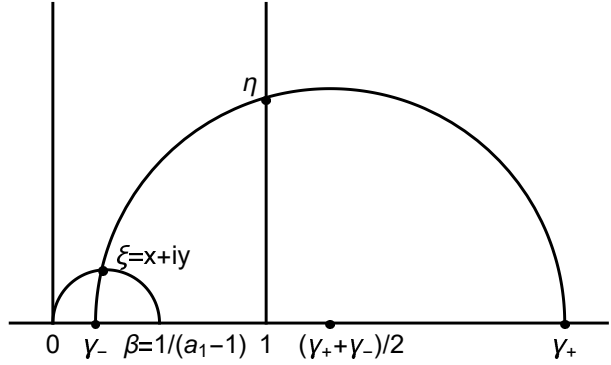}
\caption{The first return length}
\label{Figure8}
\end{figure}

Assume first that $\gamma$ is in case A, which means $\rho_o(x)=\frac{1}{a_1-x}$ and
$\xi=x+iy =[0,\beta] \cap [\gamma_-,\gamma_+]$ where $\beta=\frac{1}{a_1-1}$ and
$\eta =[1,\infty] \cap [\gamma_-,\gamma_+]$.
Trigonometry of right triangles with vertices $\gamma_-$,
$\eta$, $\gamma_+$ and $\gamma_{-},\xi,\gamma_+$ provides
\begin{equation*}
\frac{\lvert \gamma_{-}-\eta\rvert}{\lvert \gamma_+ -\eta \rvert} =
\sqrt{\frac{1-\gamma_{-}}{\gamma_+ -1}} \qquad \mbox{\rm and} \qquad
\frac{\lvert \gamma_+ -\xi \rvert}{\lvert \gamma_{-} -\xi \rvert} =
\sqrt{\frac{\gamma_+ -x}{x-\gamma_{-}}} .
\end{equation*}
The equalities $\lvert x+iy-\frac{\beta}{2}\rvert =\frac{\beta}{2}$
and $\lvert x+iy -\frac{1}{2} (\gamma_{-}+\gamma_+) \rvert =\frac{1}{2}(\gamma_+ -\gamma_{-})$ lead to
\begin{equation*}
x=\operatorname{Re} \xi =
\frac{\gamma_{-} \gamma_+}{\gamma_{-}+\gamma_+ -\beta} ,
\end{equation*}
and thus
\begin{equation*}
d_A (\xi_\gamma,\eta_\gamma) =\frac{1}{2} \log \bigg(
\frac{\gamma_+ -x}{\gamma_+-1} \cdot \frac{\gamma_- -1}{\gamma_- -x} \bigg)
=\frac{1}{2} \log \bigg( \frac{\gamma_+^2 (\beta^{-1} -\gamma_+^{-1})}{\gamma_+-1} \cdot
\frac{\gamma_- -1}{\gamma_-^2 (\beta^{-1} -\gamma_-^{-1})}   \bigg) .
\end{equation*}
Employing $\frac{1}{\beta}-\rho_o(x)=x-1$, we gather
\begin{equation*}
d_A (\xi_\gamma , \eta_\gamma) =\frac{1}{2} \log \bigg(
\frac{F_A (\gamma_\infty)}{F_A (\gamma_{-\infty})} \bigg),
\quad \mbox{\rm where} \quad
F_A (x)= \rho_o (x)^2 \, \frac{x-1}{\rho_o(x)-1} .
\end{equation*}

Similar computations in each of the cases B, C, D provide
\begin{equation*}
F_B (x)= \rho_o(x)^2\, \frac{x-1}{\rho_o(x)+1} ,\qquad
F_C (x)= \rho_o(x)^2\, \frac{x+1}{\rho_o (x)+1},\qquad
F_D (x)= \rho_o(x)^2\, \frac{x+ 1}{\rho_o(x)-1} ,
\end{equation*}
and actually we get the general formula
\begin{equation}\label{eq5.4}
d(\xi_\gamma, \eta_\gamma) =\frac{1}{2} \log \bigg(
\frac{F (\gamma_\infty)}{F (\gamma_{-\infty})}\bigg) \quad
\mbox{\rm with} \quad
F(x)= F_\gamma (x)=\rho_o(x)^2\, \frac{x-\operatorname{sign}(\gamma_\infty)}{\rho_0 (x)-\operatorname{sign}(\rho_o(\gamma_\infty))} .
\end{equation}

If $(\alpha,\beta)\in \SSS_o$, then there exists $\e\in \{\pm 1\}$ such that $(\e\alpha,\e\beta)=(\alpha^\prime,\beta^\prime)$
with
\begin{equation*}
\alpha^\prime =[\![ (a_1,\e_1);(a_2,\e_2),\ldots ]\!]_o >1,\quad\beta^\prime =- \llangle (a_0,\e_0),(a_{-1},\e_{-1}),\ldots \rrangle_o
\in -I_G.
\end{equation*}
Eqs. \eqref{eq3.5} and \eqref{eq5.4} then provide
\begin{equation}\label{eq5.5}
\begin{split}
F (\alpha) & =\rho_o(\alpha)^2 (-\e_1) \,\frac{[\![ (a_1,\e_1);(a_2,\e_2),(a_3,\e_3),\ldots ]\!]_o -1}{[\![
(a_2,\e_2);(a_3,\e_3),(a_4,\e_4),\ldots ]\!]_o -1} =\rho_o(\alpha)^2 F_+ (\alpha) , \\
F (\beta) & =
\rho_o( \beta)^2 (-\e_1) \,\frac{\llangle (a_{0},\e_{0}),(a_{-1},\e_{-1}),(a_{-2},\e_{-2}),\ldots \rrangle_o +1}{\llangle
(a_{1},\e_{1}),(a_{0},\e_{0}),(a_{-1},\e_{-1}),\ldots \rrangle_o +1} =\rho_o (\beta)^2 F_- (\beta) .
\end{split}
\end{equation}

When $\bar{\gamma}$ is a closed geodesic on $\MM_o$ with endpoints as in Proposition \ref{cor9}, the contribution of each of the factors
$F_+$ and $F_-$ to the length of $\bar{\gamma}$ is one. Using
$(a_{n+r},\e_{n+r})=(a_n,\e_n)$, we find
\begin{equation}\label{eq5.6}
\operatorname{length}(\bar{\gamma})  =
\log \Bigg( \prod\limits_{k=1}^{r}
\frac{[\![ \overline{(a_{k+1},\e_{k+1});(a_{k+2},\e_{k+2}),\ldots ,(a_{k+r},\e_{k+r})} ]\!]_o^2}{
\llangle \overline{(a_{k},\e_{k}),(a_{k-1},\e_{k-1}),\ldots, (a_{k+r-1},\e_{k+r-1})}\rrangle_o^2}\Bigg)
=\log \Bigg( \frac{(\rho_o^r)^\prime (\gamma_\infty)}{(\rho_o^r)^\prime (\gamma_{-\infty})} \Bigg) .
\end{equation}

\section{Geodesic coding and even continued fractions}

\subsection{The group $\Theta$ and the modular surface $\MM_e=\Theta\backslash\H$}

The Theta group
\begin{equation*}
\Theta :=\left\{ M\in \Gamma (1) : M \equiv I_2 \ \mbox{\rm or}\ \left( \begin{matrix}
0 & 1 \\ 1 & 0 \end{matrix}\right) \pmod{2}  \right\} ,
\end{equation*}
is the index three subgroup of $\Gamma(1)$ generated by
$S=\left( \begin{smallmatrix} 0 & -1 \\ 1 & 0 \end{smallmatrix} \right)$
and $T=\left( \begin{smallmatrix} 1 & 2 \\ 0 & 1 \end{smallmatrix} \right)$.
The image of the left half of the standard Dirichlet region
$\{ \lvert \operatorname{Re} z \rvert < 1, \lvert z\rvert >1 \}$ of
$\Theta \backslash \H$ under the transformation $S$ coincides with
the region $\{ \operatorname{Re} z >0, \lvert z\rvert <1,\lvert z-\frac{1}{2}\rvert >\frac{1}{2}\}$.
As a result, the standard Farey cell $\{ 0< \operatorname{Re} z <1, \lvert z-\frac{1}{2}\rvert >\frac{1}{2}\}$
provides a fundamental domain for $\MM_e=\Theta \backslash \H=\pi_e (\H)$.
We find that $\MM_e$ is homeomorphic to a sphere with a conic point at $\pi_e(i)$ and
cusps at $\pi_e (\infty)$ and $\pi_e(1)$.

The edges of the Farey tessellation project to
the line running from $\pi_e(1)$ to $\pi_e (\infty)$.
Bauer and Lopes \cite{BL} realized the ECF natural extension $\bT_e$ as a section of the billiard flow on $\MM_e$.
Here we describe the extension $\tilde{T}_e$ of $\bar{T}_e$ as a section
of the geodesic flow on $T_1(\MM_e)$.

The coding of geodesics on $\Theta \backslash \H$ is analogous to the coding for
$\Gamma \backslash \H$ described in Sections \ref{oddseq} and \ref{cut}. However, in this case we do not use a checkerboard coloring.
Here $\AAA_e$ is the set of geodesics in $\H$ with endpoints
\begin{equation*}
(\gplus,\gneg) \in \SSS_e := \big( (-\infty ,-1) \cup (1,\infty)\big) \times (-1,1) ,
\end{equation*}
while $\xi_\gamma$ and $\eta_\gamma$ are defined by
\begin{equation*}
\xi_\gamma =\begin{cases} \gamma \cap [1,\infty] & \mbox{\rm if $\gamma_\infty >1$} \\
\gamma \cap [-1,\infty] & \mbox{\rm if $\gamma_\infty < -1$,} \end{cases}
\qquad
\eta_\gamma =\begin{cases} \gamma \cap [a_1,a_1+\e_1] & \mbox{\rm if $\gamma_\infty >1$} \\
\gamma \cap [-a_1,-a_1 -\e_1] & \mbox{\rm if $\gamma_\infty < -1$,} \end{cases}
\end{equation*}
as in the odd continued fraction case.

Every geodesic $\bar{\gamma}$ on $\MM_e$ lifts to a geodesic $\gamma \in\AAA_e$.
Let $\XX_e$ be the set of elements $(\xi_\gamma,u_\gamma)\in T_1 (\MM_e)$ with
base point $\xi_\gamma \in \pi_e (\pm 1 +i\R)$ and unit tangent vector $u_\gamma$
pointing along the geodesic $\pi_e(\gamma)$ such that $\pi_e(\eta)$ gives the base point
of the first return of $\pi_e (\gamma)$ to $\XX_e$. The base point $\xi_\gamma$ breaks
the cutting sequence of $\pi_e(\gamma)$ into strings $L^{2k-2}R$, $L^{2k-1}R$,
$R^{2k-2}L$, $R^{2k-1}L$ that are concatenating according to the rules that will be
described in Subsection \ref{evencut}.

As in Subsection \ref{oddcut} four cases can occur, depicted in Figures \ref{Figure9} and \ref{Figure10}.
Again, we consider cases A, B, C, and D. Here, case A corresponds to $\gplus\in(2k-1,2k)$ and $\e_1=-1$,
case B to $\gplus\in(2k,2k+1)$ and $\e_1=+1$, case C to $\gplus\in(-2k,-2k+1)$ and $\e_1=-1$,
and case D to $\gplus\in(-2k-1,-2k)$ and $\e_1=+1$.

In cases A and B, we write $\gamma_\infty =[\![ (a_1,\e_1);(a_2,\e_2),\ldots ]\!]_e$ and
$\gamma_{-\infty} =- \llangle (a_0,\e_0),(a_{-1},\e_{-1}),\ldots \rrangle_e$.
For C and D, we write $\gamma_\infty =-[\![ (a_1,\e_1);(a_2,\e_2),\ldots ]\!]_e$ and
$\gamma_{-\infty} = \llangle (a_0,\e_0),(a_{-1},\e_{-1}),\ldots \rrangle_e$.

\begin{figure}
\centering\includegraphics*[scale=0.8, bb=0 0 270 120]{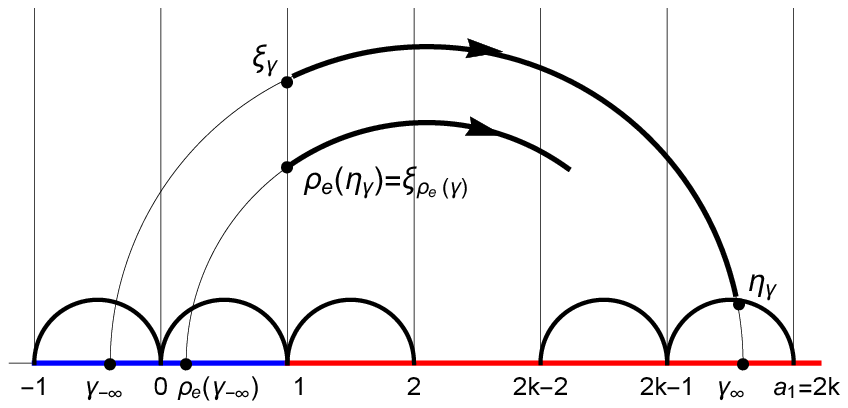}
\includegraphics*[scale =0.8, bb=0 0 250 150]{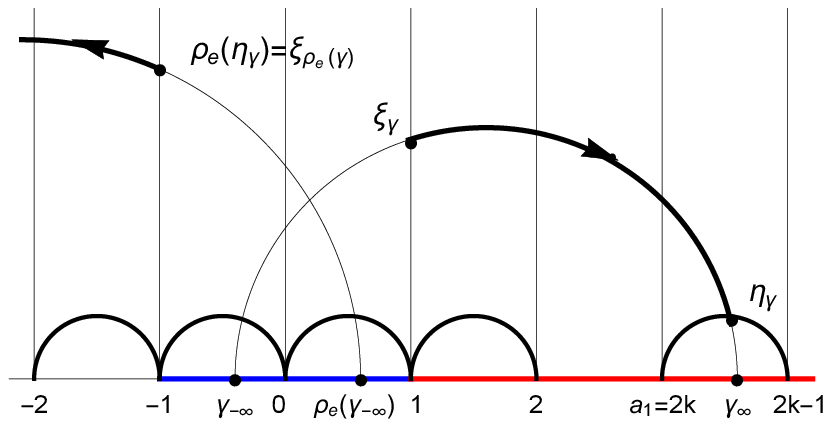}
\caption{ECF cases A ($\gamma_\infty >1,\e_1=-1,\xi_\gamma L^{2k-2}R\eta_\gamma$) and B
($\gamma_\infty >1,\e_1=+1,\xi_\gamma L^{2k-1}R\eta_\gamma$)}
\label{Figure9}
\end{figure}

\begin{figure}
\centering\includegraphics*[scale=0.8, bb=0 0 270 120]{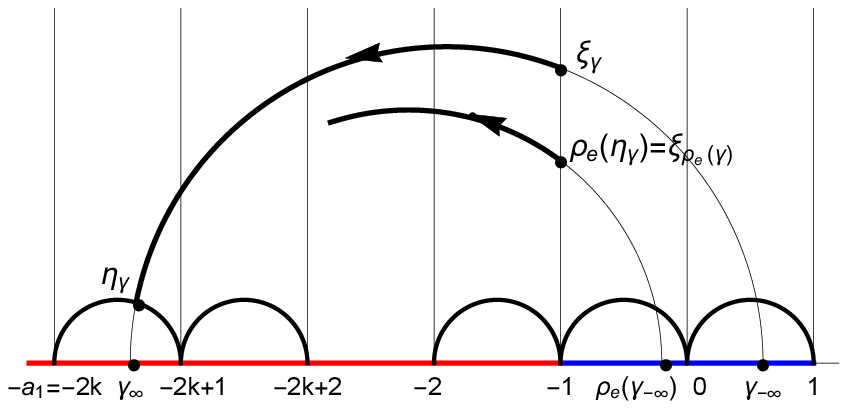}
\includegraphics*[scale =0.8, bb=0 0 250 150]{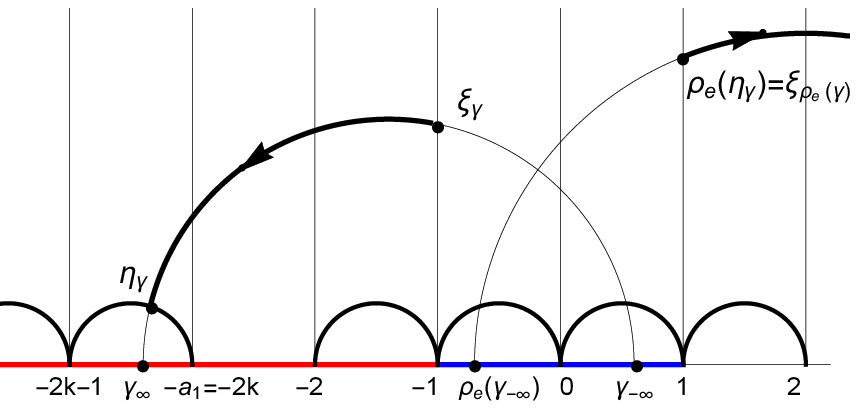}
\caption{ECF cases C ($\gamma_\infty <-1, \e_1=-1,\xi_\gamma R^{2k-2}L\eta_\gamma$) and D
($\gamma_\infty <-1, \e_1=+1,\xi_\gamma R^{2k-1}L\eta_\gamma$)}
\label{Figure10}
\end{figure}

When $\gplus>1$,
the M\" obius transformation $\rho_e(x)=\frac{1}{a_1-x}$ belongs to $\Theta$ and we have
\begin{equation*}
\rho_e (\gamma_{-\infty})  =\e_1 \llangle (a_1,\e_1),(a_0,\e_0),\ldots \rrangle_e \in (0,1) \quad
\mbox{\rm and} \quad \rho_e (\gamma_\infty)  = (-\e_1) [\![ (a_2,\e_2); (a_3,\e_3),\ldots ]\!]_e.
\end{equation*}

When $\gamma_\infty <-1$ and we have
the M\" obius transformation $\rho_e(x)=\frac{1}{-a_1-x}$ belongs to $\Theta$ and
\begin{equation*}
\rho_e (\gamma_{-\infty})  =-\e_1 \llangle (a_1,\e_1),(a_0,\e_0),\ldots \rrangle_e \in (-1,0) \quad
\mbox{\rm and} \quad \rho_e (\gamma_\infty)  = \e_1 [\![ (a_2,\e_2); (a_3,\e_3),\ldots ]\!]_e.
\end{equation*}
In all cases we have
\begin{equation}\label{eq6.1}
\begin{split}
\rho_e & \big( \e [\![ (a_1,\e_1);(a_2,\e_2),\ldots ]\!]_e , -\e
\llangle (a_0,\e_0),(a_{-1},\e_{-1}),\ldots \rrangle_e) \\
& = -\e_1 \e \big(  [\![ (a_2,\e_2);(a_3,\e_3),\ldots ]\!]_e,
- \llangle (a_{1},\e_{1}),(a_{0},\e_{0}),\ldots \rrangle_e \big) ,
\end{split}
\end{equation}
where $\e =+1$ in cases A and B, and $\e=-1$ in cases C and D.
Again, we get that when $\e_1=+1$, $\rho_e(\gplus)$ and $\gplus$ are on opposite sides of the imaginary axis,
and $\rho_e$ agrees with the transformation $\rho$ for the RCF in \cite{Ser}. As in the OCF case, this is exactly where
the RCF and ECF agree, cases A and C are followed by A or B, and cases B and D are followed by cases C or D.

The map $J_e:\SSS_e \rightarrow \tOmega_e$, $J_e (x,y)=\operatorname{sign} (x)
(\frac{1}{x},-y,1)$ is invertible. Direct verification reveals
\begin{equation}\label{eq6.2}
J_e \rho_e J_e^{-1} = \tT_e .
\end{equation}

As in Subsection \ref{invmeas}, the push-forwards of the measure
$(\alpha-\beta)^{-2}d\alpha d\beta$ on $\SSS_e$ under the maps
$\pi\circ J_e$, $\pi_1 \circ J_e$, and $\pi_2 \circ J_e$ are
$\bar{T}_e$-invariant, $T_e$-invariant,
and $\tau_e$-invariant, respectively.
For intervals $[a,b]\subset (0,1)$, $[c,d]\subset (-1,1)$ and $E=[a,b]\times [c,d]$, we find
\begin{equation*}
\begin{split}
\bar{\mu}_e (E) & =2\iint_E \frac{dxdy}{(1+xy)^2} ,
\qquad \nu_e ([c,d]) = 2\int_c^d \frac{dv}{1+v} , \\
\mu_e ([a,b]) & =\int_{1/b}^{1/a} \int_{-1}^1 \frac{d\alpha d\beta}{(\alpha-\beta)^2} +
\int_{-1/a}^{-1/b} \int_{-1}^1 \frac{d\alpha d\beta}{(\alpha-\beta)^2} =
\int_a^b \Big( \frac{1}{1+u}+\frac{1}{1-u}\Big) du ,
\end{split}
\end{equation*}
which coincide with the invariant measures from \cite{Sch2}.

The analogues of Propositions \ref{cor5}, \ref{cor6}, \ref{cor8} and
\ref{cor9} come from changing the subscript $o$ to $e$, since the same equalities hold.

\begin{proposition}\label{evencors}
{\em (i)}
A real number $\alpha >1$ has a purely periodic ECF expansion if and only if
$\alpha$ is a quadratic surd with $-1 < \balpha <1$. Furthermore, if
\begin{equation*}
\alpha =[\![ \, \overline{(a_1,\e_1); (a_2,\e_2), \ldots ,(a_{r},\e_{r})} \,]\!]_e ,
\end{equation*}
then
\begin{equation*}
\balpha=-\llangle \,\overline{(a_{r},\e_{r}), \ldots, (a_1,\e_1)} \,\rrangle_e .
\end{equation*}

{\em (ii)} Two irrational numbers $\alpha,\beta  >1$ are $\Theta$-equivalent if and only if there are
$r,s\geqslant 0$ such that
\begin{equation}\label{eventail}
t_r(\alpha)=t_s (\beta) ,
\end{equation}
where $t_m ([\![ (a_1,\e_1);(a_2,\e_2),\ldots ]\! ]_e) :=(-\e_1)\cdots (-\e_m)
[\![ (a_{m+1},\e_{m+1});(a_{m+2},\e_{m+2}),\ldots ]\!]_e$.

{\em (iii)} The ECF expansion of an irrational $\alpha$ is eventually periodic if and only if $\alpha$ is a
quadratic surd.

{\em (iv)} A geodesic $\bar{\gamma}$ on $\MM_e$ is closed if and only if
it has a lift $\gamma\in \AAA_e$ with purely
periodic endpoints
\begin{equation}\label{eq6.4}
\gamma_\infty =\e [\! [ \overline{(a_1,\e_1);(a_2,\e_2),\ldots, (a_r,\e_r)} ]\!]_e \quad \mbox{and}
\quad \gamma_{-\infty} =-\e \llangle \overline{(a_r,\e_r),\ldots ,(a_2,\e_2),(a_1,\e_1)} \rrangle_e
\end{equation}
for some $\e\in \{\pm 1\}$ and $(-\e_1)\cdots (-\e_r)=1$.
\end{proposition}

The second point should be compared with Theorem 1 in
\cite{KL}. It seems that the definition of
the tail $t_n(x)$ in \cite[Eq.\,(1.3)]{KL} should be changed to
$t_n(x)=(-e_1)\cdots (-e_n) [0;e_{n+1}/a_{n+1},e_{n+2}/a_{n+2}, \ldots]$ for that statement to hold.

Analogous formulas for the roof function
and for the length of a closed geodesic also hold, as in \eqref{eq5.4} and \eqref{eq5.5}.
Since $\llangle (b_0,\e_0),(b_1,\e_1),\ldots \rrangle_e^{-1}=\e_0
[\![ (b_0,\e_1);(b_1,\e_2),(b_2,\e_3),\ldots]\!]_e$, the analogue of
\eqref{eq5.6} shows that the length of a closed geodesic $\bar{\gamma}$ on $\MM_e$ with
endpoints as in \eqref{eq6.4} is given by
\begin{equation}\label{eq6.5}
\log \bigg( \prod\limits_{k=1}^r
[\![ \overline{(a_{k+1},\e_{k+1});\ldots,(a_{k+r},\e_{k+r})} ]\!]_e^2
[\![ \overline{(a_k,\e_{k-1});(a_{k-1},\e_{k-2}),\ldots,(a_{k-r+1},\e_{k-r})} ]\!]_e^2 \bigg).
\end{equation}

In this case $\MM_e$ has two cusps, $\pi_e (\infty)$ and $\pi_e (1)$, and the group $\Theta$ splits the rationals in two equivalence classes,
as shown by the following elementary

\begin{lemma}\label{lemma10}
$\  \Theta \infty=\big\{ \frac{a}{c}\in \Q :\mbox{$a$ or $c$ is even}\big\}$ and
$\ \Theta 1=\big\{ \frac{m}{n}\in\Q : \mbox{$m$ and $n$ are odd}\big\}$.
\end{lemma}

\subsection{Connection with cutting sequence}\label{evencut}
As before, cases A and B give the cutting sequence $\ldots xL^{n_1}R^{n_2}L^{n_3}\ldots$, where $x$ indicates $\gamma\cap [0,\infty]$ and $[n_1;n_2,n_3\ldots]$ is the regular continued fraction expansion of $\gplus$. This again corresponds to $\ldots L \x L^{n_1-1}R^{n_2}L^{n_3}\ldots$. We have two cases to consider for the first digit of the even continued fraction expansion.

\begin{description}
\item[(A)] $n_1=2k-1$ is odd, and $\gplus\in(2k-1,2k)$.
This gives the cutting sequence $\ldots\x L^{2k-2}R\eta_\gamma$ and we get $(2k,-1)=(a_1,\e_1)$. The next digit is represented by $L^nR$. 
In the case $n_2>1$ the next digit is $L^0R$, corresponding to $(2,-1)$. This again corresponds to the insertion and singularization algorithm \cite{BL}, as follows:
\begin{align*}
&2k-1+\cfrac{1}{1+\frac{1}{n_3+\ldots}}=[\![(2k,-1);(n_3+1,+1),\ldots]\!]_e, \\
&2k-1+\frac{1}{n_2+\ldots}=[\![(2k,-1);(2,-1)^{n_2-1},(n_3+1,+1),\ldots]\!]_e
\quad \mbox{\rm if $n_2 >1$,}
\end{align*}
where $(2,-1)^{t}$ means the digit $(2,-1)\ t$-times.

\item[(B)] $n_1=2k$ is even, and $\gplus\in(2k,2k+1)$.
This gives the cutting sequence $\ldots \x L^{2k-1}R\eta_\gamma$, and we get $(2k,+1)=(a_1,\e_1)$. The next digit is represented by $R^nL$. 
In the first case, we needed to consider what happened when $n_2>1,$ here we need to look at $n_2=1$. Then, we have $\ldots \x L^{2k-1}R\eta_\gamma L^{n_3}\ldots$, and the next digit corresponds to $R^0L$, which is again $(2,-1)$.
\end{description}
Strings starting with $R$ are treating similarly, with the roles of $L$ and $R$ switched.

Cases C and D are classified similarly, with roles of $L$ and $R$ exchanged. In this case, we get the cutting sequence $\ldots \x R^{n_1-1}L \eta_\gamma L^{n_2-1}R^{n_3}\ldots$, where $\gplus=-[n_1;n_2,n_3,\ldots]$.

\subsection{Extended even continued fractions} For the end point $\gneg$, we must consider the extended even continued fractions. Unlike the odd and grotesque continued fractions, the extended even continued fractions correspond to reindexing of the even continued fractions.

We consider cases A, B, C, and D similar to the grotesque continued fraction case.

\begin{description}
\item[(A)] $\gplus>1, \gneg\in(0,1)$, and $\e_0=-1$ gives
the cutting sequence $\ldots L^{2k-2}R \x\ldots$ and $(a_0,\e_0)=(2k,+1)$.
\item[(B)] $\gplus>1, \gneg\in(-1,0),$ and $\e_0=+1$ gives
the cutting sequence $\ldots R^{2k-1}L \x\ldots$ and $(a_0,\e_0)=(2k,-1)$.
\item[(C)] $\gplus<-1, \gneg\in(0, 1),$ and $\e_0=+1$ gives
the cutting sequence $\ldots L^{2k-1}R \x\ldots$ and $(a_0,\e_0)=(2k,+1)$.
\item[(D) ]$\gplus<-1,\gneg\in(-1,0),$ and $\e_0=-1$ gives
the cutting sequence $\ldots R^{2k-2}L \x\ldots$ and $(a_0,\e_0)=(2k,-1)$.
\end{description}

That is, we interpret the strings the same way as in the even continued fractions. However, strings of type A and C must be preceded by those of type A or B, and strings of type B and D must be preceded by C or D. This is similar to the restrictions on the grotesque continued fractions from Section \ref{grot}.

\section{Acknowledgments}
We are grateful to Pierre Arnoux, Byron Heersink, and to the anonymous referee for constructive comments and suggestions.

The first author acknowledges partial support during his
visits to IMAR Bucharest by a grant from Romanian Ministry of
Research and Innovation, CNCS-UEFISCDI, project PN-III-P4-ID-PCE-2016-0823, within PNCDI III.

\end{document}